\documentclass[11pt,letterpaper]{amsart}

\usepackage{graphicx}
\usepackage{euscript, epstopdf}
\usepackage{mathrsfs}
\usepackage{amssymb,amsthm}
\usepackage{txfonts}
\usepackage{amscd}
\usepackage{amsfonts,latexsym,amsmath,amsxtra,mathdots,amssymb,latexsym,mathabx,mathtools}
\usepackage{cases}
\usepackage{tikz,tikz-cd}
\usepackage{tensor,blkarray}
\usepackage{hyperref}

   \newcommand{\BC}{{\mathbb {C}}}

 \newcommand{\BR}{{\mathbb {R}}}   
    
 \newcommand{\BZ}{{\mathbb {Z}}}

   \newcommand{\frm}{\mathfrak{m}}
   
 \newcommand{\frR}{\mathfrak{R}}  \newcommand{\frS}{\mathfrak{S}}

\newcommand{\GL}{{\mathrm {GL}}} 
 
\newcommand{\SO}{{\mathrm{SO}}}
 \newcommand{\Tr}{{\mathrm{Tr}}}

\newcommand{\ind}{\mathrm{Ind}}

\newcommand{\Hom}{\mathrm{Hom}}

\newcommand*\mcapinn[2]{\vcenter{\hbox{$\mathsurround=0pt
			\ifx\displaystyle#1\textstyle\else#1\fi\bigcap$}}}

\newcommand{\ra}{\rightarrow}

\newcommand{\shskip}{\hskip 0.5 pt}

\newcommand{\pmtrix}[4]{\begin{pmatrix} #1 & #2  \\ #3 & #4 \end{pmatrix}}

\newcommand{\diag}{\textrm{diag}}    \newcommand{\supp}{\mathrm{supp}}

\makeatletter
\g@addto@macro\normalsize{\setlength\abovedisplayskip{3pt}}
\makeatother

\makeatletter
\g@addto@macro\normalsize{\setlength\belowdisplayskip{3pt}}
\makeatother

\newcommand{\delete}[1]{}

\theoremstyle{plain}

\newtheorem{thm}{Theorem}[section] 
\numberwithin{thm}{section}
   \newtheorem{lem}[thm]{Lemma}  
\newtheorem{prop}[thm]{Proposition}  \newtheorem{cor}[thm]{Corollary}
\newtheorem {rem}[thm]{Remark}   

\numberwithin{equation}{section}

\newcommand{\steinberg}{\textup{St}}\newcommand{\speh}{\text{Sp}}\newcommand{\langlands}{\textup{L}}
\usepackage{xpatch}
\makeatletter
\AtBeginDocument{\xpatchcmd{\@thm}{\thm@headpunct{.}}{\thm@headpunct{}}{}{}}
\makeatother

\begin{document}

\title{Linear Periods for Unitary Representations%\thanks{Grants or other notes
}


\author{Chang Yang      
}



\maketitle

\begin{abstract}
 Let $F$ be a local non-Archimedean field of characteristic zero with a finite residue field. Based on Tadi\'{c}'s classification of the unitary dual of $\GL_{2n}(F)$, we classify irreducible unitary representations of $\GL_{2n}(F)$ that have nonzero linear periods, in terms of Speh representations that have nonzero periods. We also give a necessary and sufficient condition for the existence of a nonzero linear period for a Speh representation.

\end{abstract}

\section{Introduction}
\label{intro}

\subsection{Main results}

Let $F$ be a local non-Archimedean field of characteristic zero with a finite residue field. Denote the group $G_n = \GL_{n}(F)$. Let $p$ and $q$ be two nonnegative integers with $p + q = n$, we denote by $H = H_{p,q}$ the subgroup of $G_n$ of matrices of the form:
\begin{align*}
	\begin{pmatrix}
		g_1 & 0 \\ 0 & g_2 
	\end{pmatrix} \quad \text{with }g_1 \in G_p,\ g_2 \in G_q.
\end{align*}
Let $\pi$ be a smooth representation of $G_n$ on a complex vector space $V$ and $\chiup$ a character of $H$, denote by $\Hom_H (\pi,\chiup)$ the space of linear forms $l$ on $V$ such that $l(\pi(h)v) = \chiup(h)l(v)$ for all $v \in V$ and $h \in H$. Smooth representations $\pi$ of $G_n$ with $\Hom_H (\pi,\chiup) \neq 0 $ are called \emph{$(H,\chiup)$-distinguished}, or simply $H$-distinguished if $\chiup$ is the trivial character $\mathbf{1}$ of $H$.

Elements of $\Hom_H(\pi,\mathbf{1})$ are called (local) linear periods of $\pi$. Linear periods have been studied by many authors. The uniqueness of linear periods was proved by Jacquet and Rallis in \cite{Jacquet-Rallis-LinearPeriods}; the uniqueness of twisted linear periods, with respect to almost all characters $\chiup$ of $H$ and in the case $p = q$, was proved by Chen and Sun in \cite{Sun-Chen-TwistedLinear+Shalika}. It thus remains an interesting question of characterizing irreducible representations that have nonzero linear periods. It is known that a tempered representation of $\GL_{2n}(F)$ has nonzero linear periods with respect to $H_{n,n}$ if and only if it is a functorial transfer of a generic tempered representation of $\SO_{2n+1}(F)$, see \cite{Jiang-Soudry-Generic+Functoriality+SO2n+1}, \cite{Matringe-Linear+Shalika-JNT} and \cite{Matringe-BF+L-function}. Another closely related characterization of the existence of nonzero linear periods for an essentially square-integrable representation is through poles of the local exterior square $L$-functions associated with the representation, see \cite{Matringe-Linear+Shalika-JNT} and references therein. A recent preprint by S\'{e}cherre \cite{Secherre-Types+LinearModel} studied supercuspidal representations with nonzero linear periods from the point of view of type theory. However, all of these characterizations are for generic representations. Motivated by the recent work of Gan-Gross-Prasad \cite{GGP--BranchingLaw+NonTempered} on branching laws in the non-tempered case, we are led to consider in this work the existence of nonzero linear periods for irreducible unitary representations. 

Our main results are as follows. We refer the reader to Section \ref{section::notation+pre} for unexplained notation in the following two theorems.

\begin{thm}\label{thm::1 Speh}
	Let $\speh(\delta,k)$ be a Speh representation of $G_{2n}$, where $\delta$ is a square-integrable representation of $G_d$ with $d > 1$, and $k$ is a positive integer ($2n = dk$). Then $\speh(\delta,k)$ is $H_{n,n}$-distinguished if and only if $d$ is even and $\delta$ is $H_{d/2,d/2}$-distinguished.
\end{thm}

\begin{thm}\label{thm::2 unitary dual}
	An irreducible unitary representation $\pi$ of $G_{2n}$ is $H_{n,n}$-distinguished if and only if it is self-dual and its Arthur part $\pi_{\textup{Ar}}$ is of the form
	\begin{align*}
		(\sigma_1 \times \sigma_1^{\vee})  \times \cdots \times (\sigma_r \times \sigma_r^{\vee})  \times \sigma_{r+1} \times \cdots \times \sigma_s.  
	\end{align*}
	where each $\sigma_i$ is a Speh representation for $i = 1, \cdots, s$, and each representation $\sigma_j$ is $H_{m_j,m_j}$-distinguished for some positive integer $m_j$, $j =  r+1 , \cdots ,s$.
\end{thm}

Distinction problem for unitary representation has already been considered by Matringe for local Galois periods in \cite{Matringe-UnitaryDual+GaloisModel} and by Offen and Sayag for local Symplectic periods in \cite{Offen-Sayag-JNT-UnitaryDual+Symplectic-1,Offen-Sayag-Uniqueness-Disjointness-JFA}. We remark that the special case of Theorem \ref{thm::2 unitary dual} for representations of Arthur type (see Theorem \ref{thm::proof+classification+UnitaryDual}) is similar to \cite[Theorem 3.13]{Matringe-BF+L-function} about local linear periods  for generic representations and the main result in \cite{Matringe-UnitaryDual+GaloisModel} about local Galois periods for unitary representations. A global analogue of our result is to find the $H_{n,n}$-distinguished representation in the automorphic dual of $G_{2n}$, which we will pursue in future works. We also refer the reader to \cite{Jacquet-Friedberg-Crell-LinearPeriods,Jacquet-Rallis-LinearPeriods} for the role of local linear periods and their global analogues in the study of standard $L$-functions.

\subsection{Remarks on the method of the proof}

Most of our work deals with distinction of parabolically induced representations of $G_n$. The main tool to study distinction of induced representations is the geometric lemma of Bernstein-Zelevinsky \cite{BZ-I}, which relates distinction of an induced representation to distinction of some Jacquet module of the inducing data. It was shown by Tadi\'{c} in \cite{Tadic-86-UnitaryDual-GLn} that every irreducible unitary representation is isomorphic to the parabolic induction of Speh representations or their twists. The observation is that Jacquet modules of Speh representations have convenient combinatorial descriptions similar to those  of Jacquet modules of essentially square-integrable representations (\cite{Lapid-Kret-JacquetModule+Ladder}). As hinted by the geometric lemma, to classify $H_{n,n}$-distinguished irreducible unitary representations, it is necessary to consider $H_{p,q}$-distinction with respect to a particular family of characters in \eqref{formula::character of Hpq}, not only of Speh representations, but also of a larger class of representations, ladder representations. The class of ladder representations was introduced by Lapid and M\'{i}nguez in \cite{Lapid-Minguez--Ladder}, and has many remarkable properties which make them an ideal testing ground for distinction of non-generic representations and some other questions in the representation theory of general linear groups, see for example \cite{FLO-Distinction+UnitaryGroup--IHES}, \cite{Gurevich-GaloisSymmetric-Ladder}, \cite{MOS-Klyachko--Ladder} and \cite{Lapid-Minguez-ParabolicInduction}. The most complicated part of the paper, Section \ref{section::ladder}, is devoted to the study of distinction of ladder representations. Our treatment is largely combinatorial based on detailed analysis by the geometric lemma. We refer the reader to \cite{Matringe-GaloisModel--Generic+Classification} and \cite{Matringe-BF+L-function} for a similar approach to the classification of distinguished generic representations in Galois symmetric space and our setting respectively. 

We next outline the proof of Theorem \ref{thm::1 Speh}. For the `if' part, the existence of non-zero linear periods for the standard module of a Speh representation $\speh(\Delta,k)$ is guaranteed by the work of Blanc and Delorme \cite{Blanc-Delorme} when $\Delta$ is $H_{d/2,d/2}$-distinguished. Thus it suffices to show that the maximal proper subrepresentation of the standard module associated with $\speh(\Delta,k)$ is not $H_{n,n}$-distinguished. The explicit structure of this maximal proper subrepresentation is well known by the work of Tadi\'{c} \cite{Tadic-95-Characters+UnitaryDual} (see also \cite{Lapid-Minguez--Ladder}). For the `only if' part of Theorem \ref{thm::1 Speh}, however, we cannot expect to get any information on the distinguishedness of $\Delta$ from that of the standard module of $\speh(\Delta,k)$ when $k$ is an even, as in this case, the standard module of $\speh(\Delta,k)$ is $H_{n,n}$-distinguished for any self-dual $\Delta$ by the work of Blanc and Delorme \cite{Blanc-Delorme}. We instead use the idea of `restricting to the mirabolic subgroup', and relate linear periods on $\speh(\Delta,k)$ with those on its highest shifted derivative, which is exactly $\speh(\Delta,k-1)$. The `only if' part is then proved by induction on $k$. We remark that the idea of exploiting the theory of derivatives in distinction problems has already appeared many times in the literature, see for example \cite{Cogdell-PS-Derivatives}, \cite{Kable-Asai-L-AJM}, \cite{Matringe-BF+L-function} and \cite{Matringe-UnitaryDual+GaloisModel}.

The paper is organized as follows. In Section \ref{section::notation+pre} we introduce notations and some preliminaries on the representation theory of general linear groups. In Section \ref{section:: some facts} we present some general facts on $(H_{p,q},\mu_a)$-distinguished representations, where $\mu_a$ is the character in \eqref{formula::character of Hpq}. In this section, we recall a result of Gan which is crucial for our combinatorial study of twisted linear periods. In Section \ref{section::Parabolic orbits of symmetric space} we give a detailed analysis of the parabolic orbits of the symmetric space involved and in Section \ref{section::consequences+Geometric Lemma} we draw some consequences of the geometric lemma. Section \ref{section::ladder} is devoted to the study of distinction of ladder representations. We then complete the classification in Section \ref{section::Classification-Unitary Dual}.

\subsection*{Acknowledgements}
The author thanks Wee Teck Gan for his kindness of sharing his paper \cite{Gan-Period+Theta} with us. He would also like to thank Dipendra Prasad for helpful comments and suggestions to improve the paper. The author thanks the referee for many remarks and suggestions and for pointing out an error in an earlier draft of the paper. The author was supported by the National Natural Science Foundation of China (no.12001191).

\section{Preliminaries}
\label{section::notation+pre}

Throughout the paper let $F$ be a local non-Archimedean field of characteristic zero with a finite residue field.

For any $ n \in \BZ_{\geqslant 0}$, let $G_n = \GL_n(F)$ and let $\mathscr R(G_n)$ be the category of smooth complex representations of $G_n$ of finite length. Denote by $\textup{Irr}(G_n)$ the set of equivalence classes of irreducible objects of $\mathscr R(G_n)$ and by $\mathscr{C}(G_n)$ the subset consisting of supercuspidal representations. (By convention we define $G_0$ as the trivial group and $\textup{Irr}(G_0)$ consists of the trivial representation of $G_0$.) Let $\textup{Irr}$ and $\mathscr C$ be the disjoint union of $\textup{Irr}(G_n)$ and $\mathscr C(G_n)$,$n \geqslant 0$, respectively. For a representation $\pi \in \mathscr R(G_n)$, we call $n$ the degree of $\pi$.

Let $\mathfrak{R}_n$ be the Grothendieck group of $\mathscr R(G_n)$ and $\frR = \oplus_{n \geqslant 0} \frR_n$. The canonical map from the objects of $\mathscr{R}(G_n)$ to $\frR_n$ will be denoted by $\pi \mapsto [\pi]$.

Denote by $\nu$ the character $\nu(g)  =  |\det g \shskip |$ on any $G_n$. (The $n$ will be implicit and hopefully clear from the context.) For any $\pi \in \mathscr R(G_n)$ and $a \in \BR$, denote by $\nu^a \pi$ the representation obtained from $\pi$ by twisting it by the character $\nu^a$, and denote by $\pi^{\vee}$ the contragredient of $\pi$. The sets $\textup{Irr}$ and $\mathscr C$ are invariant under taking contragredient. For a character $\chi$ of $F^{\times}$, define the real part $\Re (\chi)$ of $\chi$ to be the real number $a$ such that $ | \chi(z) |_{\BC} = |z|^a$, $z \in F^{\times}$, where $|\cdot|_{\BC}$ is the absolute value on $\BC$. For a subgroup $Q$ of $G_n$, denote by $\delta_Q$ the modular character of $Q$.

For two nonnegative integers $p$ and $q$ with $p + q = n$, we denote by $w_{p,q}$ the matrix
\begin{align*}
	w_{p,q} = \begin{pmatrix}
		0 & I_q  \\  I_p & 0 
	\end{pmatrix}.
\end{align*}
Let $H_{p,q}$ be the subgroup of $G_n$ as in the introduction. For $a \in \BR$, define the character $\mu_a$ of $H_{p,q}$ by
\begin{align}\label{formula::character of Hpq}
	\mu_a \left(\begin{pmatrix}
		g_1 &  \\  &  g_2
	\end{pmatrix} \right)   =  \nu^a(g_1)\nu^{-a}(g_2),\quad g_1 \in G_p,\ g_2 \in G_q.
\end{align}
(By convention we allow the case where $p$ or $q$ is zero.)

\subsection{Jacquet modules of induced representations}

The standard parabolic subgroups of $G_n$ are in bijection with compositions $ (n_1,  \cdots ,n_t)$ of $n$. The corresponding standard Levi subgroup is the group of block diagonal invertible matrices with block sizes $n_1, \cdots, n_t$. It is isomorphic to $G_{n_1} \times \cdots \times G_{n_t}$. 

Let $P = M \ltimes U$ be a standard parabolic subgroup of $G_n$ and $\sigma$ a smooth, complex representation of $M$. We denote by $\ind_P^{G_n}(\sigma)$ its normalized parabolic induction; for any standard Levi subgroup $L \subset M$, we denote by $r_{L,M}(\sigma)$ the normalized Jacquet module (see \cite[\S 2.3]{BZ-I}).

If $\rho_1, \cdots, \rho_t$ are representations of $G_{n_1}, \cdots, G_{n_t}$ respectively, we denote by 
\begin{align*}
	\rho_1 \times  \cdots \times \rho_t  
\end{align*}
the representation $\ind_P^{G_n} \sigma$ where $\sigma$ is the representation $\rho_1  \otimes\cdots \otimes \rho_t$ of $M$, where $M$ is the standard Levi subgroup of the parabolic subgroup $P$ corresponding to $(n_1,\cdots,n_t)$.

Next we briefly review the Jacquet module of a product of representations of finite length \cite[\S 1.6]{Zelevinsky-II} (or more precisely, its composition factors). Let $\alpha = (n_1,\cdots, n_t)$ and $\beta = (m_1,\cdots, m_s)$ be two compositions of $n$. For every $i \in \{ 1, \cdots, t\}$, let $\rho_i \in \mathscr R(G_{n_i})$. Denote by $\text{Mat}^{\alpha, \shskip \beta}$ the set of $t \times s$ matrices $B = (b_{i,j})$ with nonnegative integer entries such that
\begin{align*}
	\sum_{j=1}^s b_{i,j} = n_i,\quad i \in \{ 1,\cdots, t\},\quad \sum_{i=1}^t b_{i,j} = m_j, \quad j \in \{1,\cdots,s \}.
\end{align*}
Fix $B\in \text{Mat}^{\alpha,\shskip \beta}$. For any $i \in \{1,\cdots,t \}$, $\alpha_i = (b_{i,1},\cdots ,b_{i,s})$ is a composition of $n_i$ and we write the compostion factors of $r_{\alpha_i} (\rho_i)$ as 
\begin{align*}
	\sigma_i^{k} = \sigma^k_{i,1} \otimes \cdots \otimes \sigma^k_{i,s}, \quad \sigma^k_{i,j} \in \text{Irr}(G_{b_{i,j}}), \quad k \in \{ 1, \cdots, l_i\},
\end{align*}
where $l_i$ is the length of $r_{\alpha_i}(\rho_i)$. For any $ j \in \{ 1, \cdots, s\}$ and a sequence $\underline{k} = (k_1,\cdots,k_r)$ of integers such that $1 \leqslant k_i \leqslant l_i$, define
\begin{align*}
	\Sigma_j^{B,\underline{k}}    =    \sigma^{k_1}_{1,j} \times \cdots \times \sigma^k_{t,j} \in \mathscr R(G_{m_j}).
\end{align*}
Then we have
\begin{align*}
	[r_{\beta}(\rho_1 \times \cdots \times \rho_t] = \sum_{B \in \text{Mat}^{\alpha,\shskip \beta},\underline{k}}  [ \Sigma_1^{b,\underline{k}} \otimes \cdots \otimes \Sigma_s^{B,\underline{k}}].
\end{align*}

\subsection{Langlands classification}

By a segment of cuspidal representations we mean a set
\begin{align*}
	[a,b]_{\rho}  =  \{\nu^a \rho, \nu^{a+1}\rho, \cdots, \nu^b\rho  \},
\end{align*}
where $\rho \in \mathscr C$ and $a, b \in \BR$, $b- a \in \BZ_{\geqslant 0}$. The representation $\nu^a\rho \times \nu^{a+1}\rho \times \cdots \times \nu^b\rho$ has a unique irreducible quotient, which is an essentially square-integrable representaton and is denoted by $\Delta([a,b]_{\rho})$.
The map $[a,b]_{\rho} \mapsto \Delta([a,b]_{\rho})$ gives a bijection between the set of segments of cuspidal representations and the subset of essentially square-integrable representations in $\textup{Irr}$. (In what follows, for simplicity of notation, we shall use $\Delta$ to denote either a segment of cuspidal representations or the essentially square-integrable representations corresponding to it; we hope this will not cause any confusion.) We use the convention that $\Delta([a,b]_{\rho}) = 0$ if $b < a-1$ and $\Delta([a,a-1]_{\rho})  =  1$, the trivial represntation of $G_0$.

We denote the extremities of $\Delta = \Delta([a,b]_{\rho})$ by $\mathbf{b}(\Delta) = \nu^a\rho \in \mathscr C$ and $\mathbf{e}(\Delta) = \nu^b\rho \in \mathscr C$ respectively. We also write $l(\Delta) = b -a + 1$ for the length of $\Delta$. 

For $\rho \in \mathscr C$, we denote by $\BZ \shskip \rho$ the set $\{\nu^a \rho \ |\ a \in \BZ \}$ and call it the cuspidal line of $\rho$. We then transport the order and additive structure of $\BZ$ to the cuspidal line $\BZ \shskip \rho$. Thus we shall sometimes write $\nu^a\rho + b  = \nu^{a+b} \rho$ and $\nu^a \rho \leqslant \nu^b \rho $ if $a \leqslant b$, where $a$, $b$ are integers. By the contragredient of $\BZ \shskip \rho$ we mean the cuspidal line $\BZ \shskip \rho^{\vee}$.

Let $\Delta$ and $\Delta'$ be two segments. We say that $\Delta$ and $\Delta'$ are \emph{linked} if $\Delta \cup \Delta'$ forms a segment but neither $\Delta \subset \Delta'$ nor $\Delta' \subset \Delta$. If $\Delta$ and $\Delta'$ are linked and $\mathbf{b}(\Delta) = \mathbf{b}(\Delta')\nu^j$ with $j < 0$, then we say that $\Delta$ \emph{precedes} $\Delta'$ and write $\Delta \prec \Delta'$.

A multisegment is a multiset (that is, set with multiplicities) of segments. Denote by $\mathscr O$ the set of multisegements. For $\rho \in \mathscr C$, let $\mathscr O_{\rho}$ denote the multisegements such that all of its segements are contained in the cuspidal line $\BZ \shskip \rho$. An order $\frm = \{\Delta_1,\cdots, \Delta_t \} \in \mathscr{O}$ on a multisegments $\frm$ is of \emph{standard form} if $\Delta_i \nprec\Delta_j$ for all $i < j$. Every $\frm \in \mathscr O$ admits at least one standard order.

Let $\frm = \{\Delta_1, \cdots, \Delta_t  \} \in \mathscr O$ be ordered in standard form. The representation 
\begin{align*}
	\lambda(\frm)  =  \Delta_1 \times \cdots \times \Delta_t
\end{align*}
is independent of the choice of order of standard form. It has a unique irreducible quotient that we denote by $\langlands(\frm)$. The Langlands classification says that the map $\frm \mapsto \langlands(\frm)$ is a bijection between $\mathscr O$ and $\textup{Irr}$.

\subsection{Unitary dual of $G_n$}\label{section::unitary dual}

We briefly recall the classification of the unitary dual of $G_n$ by Tadi\'{c} \cite[Theorem D]{Tadic-86-UnitaryDual-GLn}. Let $\text{Irr}^u$ be the subset of unitarizable representations in $\text{Irr}$, and $\mathscr D^u$ the subset of all square-integrable classes in $\textup{Irr}^u$. Let $k$ be a positive integer, and let $\delta \in \mathscr D^u$. The repersentation 
\begin{align*}
	\nu^{(k-1)/2}\delta \times \nu^{(k-3)/2 } \delta \times \cdots \times \nu^{-(k-1)/2}\delta
\end{align*}
has a unique irreducible unitarizable quotient $\speh(\delta,k)$, called a Speh representation. 

Suppose $ 0 < \alpha < 1/2$. The representation $\nu^{\alpha} \speh(\delta,k) \times \nu^{-\alpha} \speh(\delta,k)$ is irreducible and unitarizable; we denote it by $\speh(\delta,k)[\alpha,-\alpha].$

Let $B$ be the set of all
\begin{align*}
	\speh(\delta,k), \ \speh(\delta,k)[\alpha,-\alpha],
\end{align*}
where $\delta \in \text{D}^u$, $k$ is a positive integer and $0 < \alpha < 1/2$. By \cite[Theorem D]{Tadic-86-UnitaryDual-GLn}, an irreducible representation $\pi$ is unitarizable if and only if it is of the form
\begin{align*}
	\pi_1 \times \cdots \times \pi_t, \quad \pi_i \in B, \ i=1,\cdots, t.
\end{align*}
Moreover, this expresssion is unique up to permutation. We call it a Tadi\'{c} decomposition of $\pi$.

By an irreducible representation of \emph{Arthur type}, we mean an irreducible unitary representation whose Tadi\'{c} decomposition does not involve any $\speh(\delta,k)[\alpha,-\alpha]$. For $\pi \in \text{Irr}^u$, we then have a decomposition $\pi = \pi_{\textup{Ar}} \times \pi_c$, where $\pi_{\textup{Ar}}$ is a representation of Arthur type and is called the Arthur part of $\pi$.

\section{Preliminaries on $(H_{p,q},\mu_a)$-Distinguished Representations}\label{section:: some facts}

\subsection{Basic facts}

\begin{lem}\label{lem::symmetry}
	(1)$\ $Let $\pi$ be a smooth representation of $G_n$. If  $\pi$ is $(H_{p,q},\mu_a)$-distinguished for two nonnegative integers $p$, $q$ with $p + q = n$ and $a \in \BR$, then $\pi$ is also $(H_{q,p},\mu_{-a})$-distinguished;
	
	(2)$\ $Let $\pi_1,\cdots, \pi_t \in \textup{Irr}(G_n)$. If $\pi_1 \times \cdots \times \pi_t$ is $(H_{p,q},\mu_a)$-distinguished for two nonnegative integers $p$, $q$ with $p + q = n$ and $a \in \BR$, then $\pi_t^{\vee} \times \cdots \times \pi_1^{\vee}$ is $(H_{p,q},\mu_{-a})$-distinguished.
\end{lem}
\begin{proof}
	The statement (1) follows from the fact that $\pi \cong \pi^{w_{q,p}}$. Let $\iota$ denote the involution $\iota(g) =  \prescript{t}{}{g}^{-1}$ of transpose inversion. Then (2) follows from the fact that $\pi \circ \iota  \cong  \pi^{\vee}$ for any irreducible representation $\pi$ and the fact that 
	$$(\pi_1 \times \cdots \times \pi_t) \circ \iota \cong (\pi_t \circ \iota) \times \cdots \times (\pi_1\circ \iota)$$.
\end{proof}

For representations of dimension one, we have the following simple lemma, whose proof we omit.

\begin{lem}\label{lem::one dimensional--distinction}
	Let $\chiup$ be a character of $G_n$. Assume that $\chiup$ is $(H_{p,q},\mu_a)$-distinguished for nonnegative integers $p$, $q$ with $p + q = n$ and $a \in \BR$. If $q = 0$ (resp. $p = 0$), then $\chiup$ is the character $\nu^a$ (resp. $\nu^{-a}$) of $G_n$; If $p$, $q > 0$, then $a = 0 $ and $\chiup = \mathbf{1}$, the trivial character of $G_n$.
\end{lem}

For untwisted linear periods, we have the following fundamental result due to Jacquet and Rallis \cite{Jacquet-Rallis-LinearPeriods}.

\begin{lem}
	Let $p$, $q$ be two positive integers with $p + q = n$. If $\pi\in \textup{Irr}(G_n)$, then $\dim \Hom_{H_{p,q}}(\pi,\mathbf{1}) \leqslant 1$. Furthermore, if $\dim \Hom_{H_{p,q}} (\pi, \mathbf{1}) = 1$, then $\pi \cong \pi^{\vee}$.
\end{lem}

\begin{rem}
	In this work we will not need multiplicity one results about (twisted) linear periods. However, the self-dualness property of distinguished representations is important for our applications of the geometric lemma. For example, one key ingredient is Proposition \ref{prop::starting point} which asserts self-duality for distinguished essentially square-integrable representations. In the case $p = q$, twisted linear periods have been studied by Chen and Sun in \cite{Sun-Chen-TwistedLinear+Shalika}. Their result shows that, for all but finitely many $a$,  $\dim \Hom_{H_{p,p}}(\pi,\mu_a) \leqslant 1$ for all $\pi \in \textup{Irr}(G_{2p})$. Due to the author's limited knowledge, one cannot deduce self-duality for distingsuished representations as in the untwisted case. For generic representations, however, one can deduce self-duality from a result of Gan as shown in the next subsection.
\end{rem}

\subsection{Relations with Shalika periods}

The Shalika subgroup of $G_{2n}$ is defined to be
\begin{align*}
	S_{2n} = \left\{ \left. \begin{pmatrix}
		a & b \\ 0 & a 
	\end{pmatrix}\  \right\vert a \in G_n,\ b \in M_n \right\} = G_n \ltimes N_{n,n},
\end{align*}
where $M_n$ indicates the set of $n \times n$ matrices with entries in $F$. Define a character $\psi_{S_{2n}}$ on $S_{2n}$ by
\begin{align}\label{formula::some known--character shalika subgroup}
	\psi_{S_{2n}} \left( \begin{pmatrix}
		a & \\ & a  
	\end{pmatrix}  \begin{pmatrix}
		1  &  b \\ & 1 
	\end{pmatrix}  \right)  = \psi_F(\Tr(b)),
\end{align}
where $\psi_F$ is a non-trivial character of $F$. For a smooth representation $\pi$ of $G_{2n}$, an element in $\Hom_{S_{2n}} (\pi,\psi_{S_{2n}})$ is called a local Shalika period of $\pi$.

In the untwisted case, the relation between linear periods and Shalika periods is well known (see \cite{Jiang-Nien-Qin-ShalikaModel+Functoriality} for their equivalence in the case of supercuspidal representations; see also a discussion for relatively square-integrable representations in \cite[\S 5]{Matringe-Linear+Shalika-JNT}). Using a theta correspondence approach, Gan proved the following result that relates generalized linear periods and generalized Shalika periods on $G_n$.
\begin{prop}\label{prop::Gan}
	Let $\pi$ be an irreducible generic representation of $G_{2n}$ and $\sigma$ an irreducible representation of $G_n$. One has
	\begin{align}
		\Hom_{S_{2n}} (\pi,\sigma \boxtimes \psi_{S_{2n}}) \cong \Hom_{H_{n,n}} (\pi, \sigma \boxtimes \BC),
	\end{align}
	where $\sigma \boxtimes \psi_{S_{2n}}$ is viewed as a representation of $S_{2n} = G_n \ltimes N_{n,n}$.
\end{prop}
\begin{proof}
	This is a consequence of Theorem 3.1 and Theorem 4.1 of \cite{Gan-Period+Theta}.
\end{proof}
In fact, in Theorem 3.1 of \cite{Gan-Period+Theta}, Gan obtained a statement that relates the generalized linear period of an irreducible representation to the generalized Shalika period of the big theta lift of its contragradient. We refer interested readers to the original paper of Gan for more details. What is pertinent to this work is the following simple corollary that relates as well twisted linear periods in our context to Shalika periods.
\begin{cor}\label{cor::corollary of Gan}
	Let $\pi$ be a generic representation of $G_{2n}$. The followings are equivalent:
	\begin{itemize}
		\item[(1)] $\pi$ is $(H_{n,n},\mu_a)$-distinguished for some $a \in \BR$;
		\item[(2)] $\pi$ is $(H_{n,n},\mu_a)$-distinguished for all $a \in \BR$;
		\item[(3)] $\pi$ is $(S_{2n},\psi_{S_{2n}})$-distinguished.
	\end{itemize}
	In particular, if one of these equivalent conditions holds, then $\pi$ is self-dual.
\end{cor}
\begin{proof}
    As $\pi$ is generic, its twist $\nu^a \pi$, for $a \in \BR$, is also generic. So
	\begin{align*}
		\Hom_{H_{n,n}}(\pi,\mu_a) & = \Hom_{H_{n,n}}( \pi , \nu^a \boxtimes \nu^{-a}) \cong \Hom_{H_{n,n}} (\nu^{a}\pi, \nu^{2a}\boxtimes \BC) \\
		& \cong \Hom_{S_{2n}} (\nu^{a}\pi , \nu^{2a} \boxtimes \psi_{S_{2n}})  \cong \Hom_{S_{2n}} (\pi , \psi_{S_{2n}}).
	\end{align*}
\end{proof}

%\begin{rem}
%	The result of Gan, Proposition \ref{prop::Gan}, highlights the importance of determining whether $\Theta(\pi) = \pi^{\vee}$ for an irreducible %representation $\pi$ of $G_n$. This is true for generic representations by a result of Minguez \cite{Minguez-HoweDuality-II}, see also \cite[Theorem %4.1]{Gan-Period+Theta}. Furthermore, in \cite{Sun-Xue-Fang-GJ+BigTheta}, the authors prove that if the Godement-Jacquet $L$ function $L(s,\pi)$ or %$L(s,\pi^{\vee})$ has no pole at $s = 1/2$, then $\Theta(\pi) \cong \pi^{\vee}$.
%\end{rem}

%\begin{rem}
%	When $n = 2$, this corollary can be deduced from a result of Waldspurger on toric periods, which shows that a generic representation of $\GL_2(F)$ is %$(T,\mu_a)$-distinguished if and only if its central character is trivial, where $T$ is the diagonal torus in $\GL_2(F)$.
%\end{rem} 

\subsection{The theory of Bernstein-Zelevinsky derivatives}
Let $P_n \subset G_n$ be the mirabolic subgroup of $G_n$ consisting of matrices with the last row $(0,0,\cdots,0,1)$. We refer the reader to \cite[3.2]{BZ-I} for the definition of the following functors 
\begin{align*}
	&\Psi^-: \textup{Alg}\,P_n \ra \textup{Alg}\, G_{n-1},  \qquad  \Psi^+ : \textup{Alg}\, G_{n-1} \ra \textup{Alg}\,P_n,   \\
	&\Phi^-: \textup{Alg}\, P_n \ra \textup{Alg}\, P_{n-1}, \qquad  \Phi^+ : \textup{Alg}\, P_{n-1} \ra \textup{Alg}\,P_n.
\end{align*}
Define $\pi^{(k)}  =  \Psi^-(\Phi^-)^{k-1}\left(\pi|_{P_n} \right)$ to be the $k$-th derivative of a representation $\pi$ of $G_n$.

The following proposition can be proved by the same argument as those in \cite[Proposition 1]{Kable-Asai-L-AJM} (see also \cite[Proposition 3.1]{Matringe-Linear+Shalika-JNT}, where the linear subgroups $H_{p,q}$ take different forms.)

\begin{prop}\label{prop::Derivative--MainFormula}
	If $\sigma$ is a representation of $P_{n-1}$ and $\chiup$ is a character of $H_{p,q}$, then
	\begin{align*}
		\Hom_{P_n \cap H_{p,q}} (\Phi^+ \sigma , \chiup) \cong \Hom_{P_{n-1} \cap H_{q-1,p}} (\sigma , \chiup^{w_{q-1,p}}\mu_{-1/2})
	\end{align*}
	as complex vector spaces, where $\chiup^{w_{q-1,p}}$ is the character of $H_{q-1,p}$ defined by $\chiup^{w_{q-1,p}}(g) = \chiup (w_{q-1,p} \shskip g \shskip w_{q-1,p}^{-1})$. In particular, for all $a \in \BR$, one has 
	\begin{align}\label{formula::Derivative-MainFormula}
		\Hom_{P_n \cap H_{p,q}} (\Phi^+ \sigma , \mu_a) \cong \Hom_{P_{n-1} \cap H_{q-1,p}} (\sigma , \mu_{-a-1/2}).
	\end{align}
\end{prop}
As a corollary, we have the following result due to Matringe \cite[Theorem 3.1]{Matringe-Linear+Shalika-JNT}.
\begin{cor}\label{cor::discrete series--max Levi}
	Let $\Delta$ be an essentially square-integrable representation of $G_n$. Let $p$, $q$ be two positive integers with $p + q =n$, and $\chiup$ a character of $H_{p,q}$. Assume that $\pi$ is $(H_{p,q},\chiup)$-distinguished. Then $p=q$.
\end{cor}
Another application of Proposition \ref{prop::Derivative--MainFormula} will generalize Corollary \ref{cor::discrete series--max Levi} to essentially Speh representations in Corollary \ref{cor::Speh + distinguishe = max Levi G_m,m} of Section \ref{section::Speh--MaxLevi}. 

As a direct consequence of Corollary \ref{cor::discrete series--max Levi} and Corollary \ref{cor::corollary of Gan}, we have:
\begin{prop}\label{prop::starting point}
	Let $\Delta$ be an essentially square-integrable representation of $G_n$. If $\Delta$ is $(H_{p,q},\mu_a)$-distinguished for two positive integers $p$, $q$ with $p + q = n$ and some $a \in \BR$, then $p = q$ and $\Delta$ is $H_{p,p}$-distingusihed (hence self-dual).
\end{prop}

\section{Symmetric Spaces and Parabolic Orbits}\label{section::Parabolic orbits of symmetric space}

The main tool we use to classify distinguished unitary representations is the geometric lemma of Bernstein and Zelevinsky \cite[Theorem 5.2]{BZ-I}. Applying it requires a detailed analysis of the double coset space $P \backslash G_n / H_{p,q}$, where $P$ is a parabolic subgroup of $G_n$. As $H_{p,q}$ is a symmetric subgroup of $G_n$, we follow the framework given by Offen in \cite{Offen-ParabolicInduction-JNT}. 

\subsection{General notations}

Let $G = G_{n}$, $H = H_{p,q}$ be the subgroup of $G_n$ as in the introduction. Let 
\begin{align*}
	\varepsilon = \varepsilon_{p,q} = \left( \begin{matrix}
		I_p &   \\   & -I_q \end{matrix} \right),
\end{align*}
and $\theta = \theta_{p,q}$ be the involution on $G_n$ defined by $\theta (g) = \varepsilon  \shskip g \shskip \varepsilon^{-1}$. The symmetric space associated to $(G,\theta)$ is 
\begin{align*}
	X = \{ g \in G \ |\ \theta (g) = g^{-1} \},
\end{align*}
equipped with the $G$-action $g \cdot x = g x \theta(g)^{-1}$. The map $g \mapsto g \cdot e$ gives a bijection of the coset space $G/H$ onto the orbit $G \cdot e \subset X$, and thus a bijection of the double coset space $P \backslash G / H$ onto the $P$-orbits in $G \cdot e$, where $P$ denotes a parabolic subgroup of $G$. For any $g \in G$, denote by $[g]_G$ the conjugacy class of $g$ in $G$. Note that the map $g \mapsto g \varepsilon$ gives a bijection of $G \cdot e $ onto $[\varepsilon]_G$ and that the $G$-action on $G \cdot e$ is transformed to the conjugation action of $G$ on $[\varepsilon]_G$. 

For any subgroup $Q $ of $G$ and $x \in X$, let $Q_x = \{ g \in Q \ |\ g \cdot x = x \}$ be the stabilizer of $x$ in $Q$. Note that $Q_x$ is just the centralizer of $x \varepsilon$ in $Q$.

\subsection{Twisted involutions in Weyl groups}

A first coarse classification of the double cosets in $P \backslash G / H$ is given by certain Weyl group elements. Let $W$ be the Weyl group of $G$. Let
\begin{align*}
	W[2] = \{w \in W\ |\ w \theta(w) = e \} = \{w \in W \ |\ w^2 = e  \}
\end{align*}
be the set of twisted involutions in $W$. For two standard Levi subgroups $M$ and $M'$ of $G$, let $\tensor[_M]{W}{_{M'}}$ be the set of all $w \in W$ that are left $W_M$-reduced and right $W_{M'}$-reduced.

Given a standard parabolic subgroup  $P = M \ltimes U$, define a map
\begin{align}\label{formula::parabolic orbit-Weyl}
	\iota_M \colon P\backslash X \rightarrow W[2] \cap \tensor[_M]{W}{_M}
\end{align}
by the relation 
\begin{align}
	PxP = P \shskip \iota_M(P \cdot x) \shskip P.
\end{align}

For $x \in X$, let
\begin{align*}
	w = \iota_M(P \cdot x)  \text{\quad and \quad} L = M(w) = M \cap wMw^{-1}.
\end{align*}
Then $L$ is a standard Levi subgroup of $M$ satisfying $L = w L w^{-1}$.

\subsection{Admissible orbits}\label{section::Admissible orbits}

It is noted in \cite{Offen-ParabolicInduction-JNT} that, to apply the geometric lemma in particular cases, it is necessary to first understand the admissible orbits. Recall that $x \in X$ (or a $P$-orbit $P \cdot x$ in X) is said to be $M$-admissible if $ M =  w M w^{-1}$ where $w = \iota_M(P \cdot x)$. We now describe the relevant data for $M$-admissible $P$-orbits in $G \cdot e$.

By \cite[Corollary 6.2]{Offen-ParabolicInduction-JNT}, $M$-admissible $P$-orbits in $G \cdot e$ is in bijection with $M$-orbits in $G \cdot e \cap N_G(M)$, or equivalently $M$-conjugacy classes in $[\varepsilon]_G \cap N_G(M)$.

% We now explicate the orbits and modular characters in Lemma \ref{lem::Distinction-Necessary-cuspidal} in our specific context.

Fix a composition $\bar{n}=(n_1,\cdots,n_t)$ of $n$. Let $P = M \ltimes U$ be the standard parabolic subgroup of $G_n$ associated to $\bar{n}$. Denote by $\frS_t^{(\bar{n})}$ the set of permutations $\tau$ on the set $\{ 1, 2, \cdots ,t\}$ such that $n_i = n_{\tau(i)}$ for all $i \in \{1, \cdots ,t \}$. To each  $\tau$ in $\mathfrak{S}_t^{(\bar{n})}$, we associate a block matrix $w_{\tau}$ which has $I_{n_i}$ on its $(\tau(i),i)$-block for each $i$ and has $0$ elsewhere. Then the map 
\begin{align*}
	\tau \mapsto w_{\tau}M
\end{align*}
defines an isomorphism of groups from $\frS_t^{(\bar{n})}$ to $N_G(M) / M$. Write an element of $M$ as  $\diag\{A_1,\cdots,A_t\}$. Note that an element $w_{\tau}\shskip \diag\{A_1,\cdots,A_t \}$ of $N_G(M)$ has order $2$ if and only if 
\begin{align*}
	\tau^2 = 1 \quad \text{and} \quad A_iA_{\tau(i)} = I_{n_i}~\ \text{for all } i \in \{ 1,\cdots ,t \}.
\end{align*}
One sees that the $M$-conjugacy classes in $[\varepsilon]_{G} \cap N_G(M)$ are parameterized by the set of pairs $(\mathfrak{c}_{\tau},\tau)$ where $\tau \in \frS_t^{(\bar{n})}$, $\tau^2 = 1$, and $\mathfrak{c}_{\tau}$ is a set of the form
\begin{align*}
	\{(n_{k,+}, n_{k,-}) \ |\ \text{for all $k$ such that $\tau(k) = k$} \}
\end{align*}
such that
\begin{align}
	\begin{cases*}
		n_k = n_{k,+} + n_{k,-},\ n_{k,+},n_{k,-} \geqslant 0;\\
		\sum_{k,\tau(k)= k} n_{k,+} + \sum_{(i,\tau(i)),i < \tau(i)} n_i  =  p;\\
		\sum_{k,\tau(k)= k} n_{k,-} + \sum_{(i,\tau(i)),i < \tau(i)} n_i  =  q.
	\end{cases*}
\end{align}
Denote by $\mathcal{I}_{p,q}^{\sharp}(\bar{n})$ the set of all such pairs.

For the $M$-admissible $P$-orbit $\mathcal{O}$ corresponding to $(\mathfrak{c}_{\tau},\tau)$ in $\mathcal{I}_{p,q}^{\sharp}(\bar{n})$, we can choose a natural orbit representative $ x = x_{(\mathfrak{c}_{\tau},\tau)} \in \mathcal{O} \cap N_G(M)$ as follows: The matrix $x \shskip \varepsilon$ has $I_{n_i}$ on its $(\tau(i),i)$-block when $\tau(i) \neq i$, $\diag(I_{n_{i,+}}, -I_{n_{i,-}})$ on its $(i,i)$-block when $\tau(i) = i$, and $0$ elsewhere. One sees easily that $M_x$ consists of elements $\diag\{A_1,\cdots,A_t  \}$ such that
\begin{align}
	\begin{cases}
		A_i  =   A_{\tau(i)},    &  \tau(i) \neq i; \\
		A_i I_{n_{i,+},n_{i,-}}  =  I_{n_{i,+},n_{i,-}} A_i &  \tau(i) = i. 
	\end{cases}
\end{align}
Here and in what follows, we denote by $I_{n_1,n_2}$ for the diagonal matrix $\diag\{I_{n_1}, -I_{n_2} \}$. Thus, when $\tau(i) = i$, we may further write $A_i$ as $\diag\{A_{i,+},A_{i,-} \}$. One also has $P_x  =  M_x \ltimes U_x$. 

The following computation of modular characters is indispensable for applications of the geometric lemma, see \cite[Theorem 4.2]{Offen-ParabolicInduction-JNT}. We omit the proof here as it is obtained by a routine calculation.

\begin{lem}\label{lem::modulus character-Admissible orbit}
	Let $x \in G\cdot e \cap N_G(M)$ be the representative as above of the $M$-admissible $P$-orbit corresponding to $(\mathfrak{c},\tau) \in \mathcal{I}_{p,q}^{\sharp}(\bar{n})$. Then, for $m = \textup{diag}\{A_1,\cdots,A_t\} \in M_x$, we have
	\begin{align}
		\delta_{P_x}\delta_P^{-1/2}(m) = \prod_{\stackrel{i<j}{\tau(i) = i, \tau(j) = j}}& \nu(A_{i,+})^{(n_{j,+} - n_{j,-})/2}  \nu(A_{i,-})^{(n_{j,-} - n_{j,+})/2}   \nu(A_{j,+})^{(n_{i,-} - n_{i,+})/2}   \\
		&  \cdot  \nu(A_{j,-})^{(n_{i,+} - n_{i,-})/2}  \prod_{\stackrel{i < j}{\tau(i) > \tau(j)}} \nu(A_i)^{-n_j/2} \nu(A_j)^{n_i/2}.  \nonumber
	\end{align}
\end{lem}

%\begin{cor}
%Let $x \in G \cdot e \cap N_G(M)$ be $M$-admissible. Let $\rho_i$ be a representation of $\GL_{n_i}(F)$, $i = 1,\cdots,t$. The %representation $\rho = \rho_1 \otimes \cdots \otimes \rho_t$ of $M$ is $(M_x,\delta_x)$-distinguished if and only if 
%\begin{align}
%\begin{cases}
%\rho_i  \cong   \rho_{\tau(i)}^{\vee}    &  \tau(i) \neq i \\
%\rho_i \text{ is }(\GL_{n_{i,+}} \times \GL_{n_{i,-}},\chi_i) \text{-distinguished} &  \tau(i) = i 
%\end{cases}
%\end{align}
%where the character $\chi_i$ is given by
%\begin{align}
%\xi_i (\diag\{ A_1,A_2\}) = \lvert \det A_1 \rvert^{\alpha_i} \lvert \det A_2 \rvert^{-\alpha_i}
%\end{align}
%and
%\begin{align*}
%\alpha_i = \frac{n_{i,+}-n_{i,-}}{2} + \sum_{\stackrel{k > i}{ \tau (k) = k } }(n_{k,+}-n_{k,-}).
%\end{align*}
%\end{cor}

\subsection{General orbits}\label{section::general orbits}

For our purposes, we consider only $P$-orbits in $G \cdot e \subset X$ where $P$ is a maximal parabolic subgroup. Let $P = P_{k,n-k}$ be the standard parabolic subgroup associated to $(k,\shskip n-k)$ with $M$ its Levi subgroup. We follow the geometric method as in \cite{Matringe-BF+L-function}. The case where $|p-q|\leqslant 1$ can be essentially covered by the results there. We remark however that the symmetric subgroup $H$ there takes a different form and the treatment here is independent.

Let $V$ be a $n$-dimensional $F$-vector space with a basis $\{e_1,\cdots, e_{n}\}$. Let $V_+$ (resp. $V_-$) be the subspace of $V$ of dimension $p$ (resp. $q$) which is generated by $\{e_1,\cdots,e_p \}$ (resp. $\{ e_{p+1},\cdots, e_n \}$). The coset space $G/ P$ can be identified with the set of subspaces of $V$ of dimension $k$. For such a subspace $W$, set 
$$r_W = \dim_F (W \cap V_+),\quad s_W = \dim_F (W \cap V_-).$$

\begin{lem}\label{lem::H-orbit in Gr(k)}
	Let $W_1$ and $W_2$ be two subspaces of $V$ of dimension $k$. Then they are in the same $H$-orbit if and only if $r_{W_1} = r_{W_2}$ and $s_{W_1}  =  s_{W_2}$. For a pair of nonnegative integers $(r,s)$, there is a subspace $W$ of $V$ such that $r = r_W$ and $s = s_W$ if and only if 
	\begin{align}\label{formula::parameterization of double cosets}
		\begin{cases*}
			r + s \leqslant k, \\
			k - s \leqslant p, \quad k -r  \leqslant q.
		\end{cases*}
	\end{align}
	
\end{lem}

Denote by $\mathcal{I}^{k}_{p,q}$ the set of pairs of nonnegative integers $(r,s)$ that satisfying \eqref{formula::parameterization of double cosets}. Then, by Lemma \ref{lem::H-orbit in Gr(k)}, the double cosets in $H \backslash G / P$ can be parameterized by $\mathcal{I}^k_{p,q}$. For $(r,s) \in \mathcal{I}^{k}_{p,q}$, call $d  =  k - r -s$ the defect of $(r,s)$.

We first seek a complete set of representatives of $P \backslash G  / H$. We split the discussions into two cases.

Case $k \geqslant p$. Let $W_{(r,s)}$ be the subspace of $V$ generated by 
\begin{align*}
	\{ e_1,\cdots,e_r;e_{r+1}+e_{q+r+1},\cdots, e_{k-s}+e_{q+k-s};e_{q+k-s+1},\cdots e_n;e_{p+1},\cdots,e_k  \}.
\end{align*}
Then $\dim_F W_{r,s} = k$, $\dim_F (W_{(r,s)} \cap V_+)  =  r$ and $\dim_F (W_{(r,s)} \cap V_-)  =  s$. Let $\tilde{\eta}_{(r,s)}^{-1}$ be the block matrix
\begin{align*}
	\pmtrix{C_1}{C_2}{C_3}{C_4}
\end{align*}
where $C_1$ and $C_4$ are matrices of size $p \times p$ and $q \times q$ respectively, and 
\begin{align*}
	& C_1  =  \pmtrix{I_{k-s}}{}{}{0},\quad C_4 = \pmtrix{I_{k-s+q-p}}{}{}{0}   \\
	& C_2  =  \pmtrix{0}{0}{0}{I_{s+p-k}},\quad C_3  = \pmtrix{0}{0}{0}{I_{p-r}}.
\end{align*}
Then $\{\tilde{\eta}_{(r,s)}^{-1} \}$ is a complete set of representatives of the double coset space $H \backslash G / P$. Taking inverse, we thus get a complete set of representatives $\{ \tilde{\eta}_{(r,s)}\}$ of $P \backslash G / H$.

Case $k \leqslant p$. Let $W_{(r,s)}$ be the subspace of $V$ of dimension $k$ generated by 
\begin{align*}
	\{e_1,e_2,\cdots, e_r;e_{r+1}+e_{n-k+r+1},\cdots, e_{k-s}+e_{n-s}; e_{n-s+1},\cdots,e_n \}.
\end{align*}
Then $\dim_F W_{r,s} = k$, $\dim_F (W_{r,s} \cap V_+) = r$ and $\dim_F (W_{(r,s)} \cap V_-)  =  s$. Let $\tilde{\eta}_{(r,s)}^{-1}$ be the block matrix
\begin{align*}
	\pmtrix{D_1}{D_2}{D_3}{D_4}
\end{align*}
where $D_1$ and $D_4$ are matrices of size $k \times k$ and $(n-k) \times (n-k)$ respectively, and 
\begin{align*}
	& D_1  =  \pmtrix{I_{k-s}}{}{}{0},\quad D_4 = \pmtrix{I_{n-k-s}}{}{}{0}   \\
	& D_2  =  \pmtrix{0}{0}{0}{I_{s}},\quad D_3  = \pmtrix{0}{0}{0}{I_{k-r}}.
\end{align*}
Then $\{\tilde{\eta}_{(r,s)}^{-1} \}$ is a complete set of representatives of the double coset space $H \backslash G / P$. Taking inverse, we thus get a complete set of representatives $\{ \tilde{\eta}_{(r,s)}\}$ of $P \backslash G / H$.

We then describe the relevant data for these general $P$-orbits in $G \cdot e$. For $(r,s) \in \mathcal{I}^k_{p,q}$, let $\tilde{x}_{(r,s)} =  \tilde{\eta}_{(r,s)} \theta(\tilde{\eta}_{(r,s)})^{-1}  \in G \cdot e $. Thus $\{ \tilde{x}_{(r,s)} \}$ is a complete set of representatives of $P$-orbits in $G\cdot e$. Write $w_{(r,s)}  =  \iota_{M} (P \cdot \tilde{x}_{(r,s)})$. Recall that $w_{(r,s)}$ is left and right $W_{M}$-reduced. In either case, we have that
\begin{align}
	w_{(r,s)}  = \left( \begin{matrix}
		I_{k-d}  &  &  &   \\  &  &  I_d & \\  & I_d &  &  \\  & & & I_{n-k-d}
	\end{matrix}\right).
\end{align}
Thus $L = L_{(r,s)} =  M \cap w_{(r,s)} M w_{(r,s)}^{-1}$ is the standard Levi subgroup associated to the composition $(k-d,\shskip d,\shskip d,\shskip n-k-d)$ of $n$. Denote by $Q$ the standard parabolic subgroup of $G_n$ with Levi subgroup $L$. We can choose, in either case, an orbit representative $x_{(r,s)} \in P \cdot \tilde{x}_{(r,s)} \cap L w_{(r,s)}$ such that
\begin{align}\label{formula::orbit representative-x(r,s)}
	x_{(r,s)} \varepsilon  =  \left( \begin{matrix}
		I_{r,s} &  & &  \\  &  &  I_d  &  \\ & I_d &  &  \\  &  &  &  I_{p+s-k, q+r-k}
	\end{matrix} \right).
\end{align}
So the group $L_{x_{(r,s)}}$ consists of elements $\diag\{A_{1,+},A_{1,-},A_2,A_3,A_{4,+},A_{4,-} \}$ such that
\begin{align*}
	\begin{cases}
		A_{1,+} \in G_r, \ A_{1,-} \in G_s,\quad A_{4,+} \in G_{p+s-k},\ A_{4,-} \in G_{q+r-k}; \\
		A_2  =  A_3 \in G_d.
	\end{cases}
\end{align*}
We can also choose $\eta_{(r,s)} \in G_n$ such that $\eta_{(r,s)} \theta (\eta_{(r,s)})^{-1}  =  x_{(r,s)}$ and that 
%\begin{align}
%  \eta_{(r,s)}^{-1}  \textup{diag}\{ A_{1,+},A_{1,-},A_2,A_3,A_{4,+},A_{4,-} \}  \eta_{(r,s)}  \\
%  =  \textup{diag}\{A_{1,+}, A_2, A_{4,+}, A_{4,-}, A_2, A_{1,-}  \} \in H_{p,q}.
%\end{align}
\begin{align}\label{formula::eta-condition}
	\eta_{(r,s)}^{-1}  \left( \begin{smallmatrix}
		A_{1,+} & & & & &  \\ & A_{1,-}& & & & \\ & & A_2& & & \\ & & & A_3 & &  \\ & & & & A_{4,+}& & \\ & & & & &A_{4,-}  \end{smallmatrix}  \right) \eta_{(r,s)}  
	=        \left( \begin{smallmatrix}
		A_{1,+} & & & & &  \\ & A_2& & & & \\ & & A_{4,+}& & & \\ & & & A_{4,-} & &  \\ & & & & A_3& & \\ & & & & &A_{1,-}  \end{smallmatrix}  \right)    \in H_{p,q}                                  
\end{align}

The modular characters for general orbtis that are relavent to us are computed as follows.

\begin{lem}\label{lem::modular character--general orbtis}
	For $(r,s) \in \mathcal{I}^k_{p,q}$, let $x = x_{(r,s)}$, $\eta = \eta_{(r,s)}$, $L$ and $Q$ as given above. For $a \in \BR$, let $\mu_a$ be the character of $H = H_{p,q}$ defined in \eqref{formula::character of Hpq}.
	
	 For
	\begin{align*}
		m = \textup{diag}\{A_{1,+},A_{1,-},A_2,A_3,A_{4,+},A_{4,-} \}  \in L_x,
	\end{align*}
	then
	\begin{align}\label{formula::modulus character--x(r,s)}
		\delta_{Q_x}\delta_Q^{-1/2} (m) & =  \nu(A_{1,+})^{(p-q+s-r)/2} \nu(A_{1,-})^{(q-p+r-s)/2} \nu(A_{4,+})^{(s-r)/2} \nu(A_{4,-})^{(r-s)/2}, \\
		\mu_a^{\eta^{-1}} (m)             &=  \nu( A_{1,+} )^{a} \nu( A_{1,-})^{-a} \nu( A_{4,+})^{a} \nu( A_{4,-})^{-a}.  \nonumber
	\end{align}
\end{lem}
\begin{proof}
	Note that $x_{(r,s)}$ is the natural representative for an $L$-admissible $Q$-orbit in $G\cdot e$ chosen in Section \ref{section::Admissible orbits}. Then \eqref{formula::modulus character--x(r,s)} follows directly from Lemma \ref{lem::modulus character-Admissible orbit}.
\end{proof}

%	where the characters $\chiup_1$ and $\chiup_4$ are given by
%	\begin{align*}
%	\chiup_1 (\diag\{ A_+,A_-\}) &= \lvert \det A_+ \rvert^{ \frac{p-q+s-r}{2}+a} \lvert \det A_- \rvert^{-\frac{p-q+s-r}{2}-a},\quad  A_+ \in G_r, A_- \in G_s ;          %\\
%	\chiup_4 (\diag\{ A_+,A_-\}) &= \lvert \det A_+ \rvert^{ \frac{s-r}{2}+a} \lvert \det A_- \rvert^{-\frac{s-r}{2}-a},\  A_+ \in G_{p+s-k}, A_- \in G_{q+r-k}.
%	\end{align*}

\section{Consequences of the Geometric Lemma}\label{section::consequences+Geometric Lemma}

\subsection{The geometric lemma}

We first recall the formulation of the geometric lemma of Bernstein and Zelevinsky in \cite[Theorem 4.2]{Offen-ParabolicInduction-JNT}, and we refer the reader to \emph{loc.cit} for unexplained notation.

\begin{prop}
	Let $P = M \ltimes U$ be a standard parabolic subgroup of $G$. Let $\sigma$ be a representation of $M$, and $\chiup$ a character of $H$. If the representation $\ind_P^G (\sigma)$ is $(H,\chiup)$-distinguished, then there exist a $P$-orbit $\mathcal{O}$ in $P \backslash (G \cdot e)$ and $\eta \in G$ satisfying $x = \eta \cdot e \in \mathcal{O} \cap Lw$ (where $w = \iota_M(P \cdot x)$ and $L = M(w)$) such that the Jacquet module $r_{L,M}(\sigma)$ is $(L_x, \delta_{Q_x}\delta_Q^{-1/2}\chiup^{\eta^{-1}})$-distinguished. Here $Q  =  L \ltimes V$ is the standard parabolic subgroup of $G$ with Levi subgroup $L$.
\end{prop}

We retain the notation of Section \ref{section::Parabolic orbits of symmetric space}. As a consequence of the orbit analysis there, we formulate the following corollary.

\begin{cor}\label{cor::Distinction--maximal parabolic}
	Let $\sigma_1$ resp. $\sigma_2$ be a representation of $G_k$ resp. $G_{n-k}$. If the representation $\sigma_1 \times \sigma_2$ is $(H_{p,q},\mu_a)$-distinguished for some $p$, $q \geqslant 0$, $p + q =n$ and $a \in \BR$, then there exists a pair $(r,s) \in \mathcal{I}^k_{p,q}$ with defect $d = k - r -s$ such that the representation $r_{(k-d,d)} \sigma_1 \otimes r_{(d,n-k-d)}\sigma_2$ of $L$ is $(L_x, \delta_{Q_x}\delta_Q^{-1/2}\mu_a^{\eta^{-1}})$-distinguished, where $L$ is the standard Levi subgroup of $G_n$ associated to $(k-d,d,d,n-k-d)$, $Q$ is the standard parabolic subgroup with $L$ its Levi part, $x = x_{(r,s)}$ is given in \eqref{formula::orbit representative-x(r,s)} and $\eta = \eta_{(r,s)} \in G_n$ such that $x = \eta \cdot e$ and \eqref{formula::eta-condition} holds. 
\end{cor}

Often in practice there is a filtration of the Jacquet module of the inducing data whose successive factors are pure tensor representations. The following lemma is a direct consequence of Lemma \ref{lem::modular character--general orbtis}.

\begin{lem}\label{lem::distinction--pure tensor}
	Notation being as above. Let $\rho  =  \rho_1 \otimes \rho_2 \otimes \rho_3 \otimes \rho_4$ be a pure tensor representation of $L$. Then $\rho$ is $(L_{x},\delta_{Q_x}\delta_Q^{-1/2}\mu_a^{\eta^{-1}})$-distinguished if and only if 
	\begin{align}
		\begin{cases}\label{formula::necessary condition--GL}
			\rho_2  \cong \rho_3^{\vee},           \\
			\rho_1 \text{ is }(H_{r,s},\shskip \mu_{a+(p - q + s - r)/2}) \text{-distinguished},  \\
			\rho_4 \text{ is }(H_{p+s-k,q+r-k}, \mu_{a+(s-r)/2})  \text{ distinguished}.      
		\end{cases}
	\end{align}
\end{lem}
\vspace{1pt}
\begin{rem}
    Our proof of classification has an inductive structure. This necessary conditions \eqref{formula::necessary condition--GL} is the reason why we study $(H_{p,q},\mu_a)$-distinction from the beginning, although our main concern is about $H_{p,p}$-distinction.	
\end{rem}

\begin{rem}
	The subscripts in the pair $(H_{p,q},\mu_a)$ play a subtle role in this work as, for example, seen from Proposition \ref{prop::left aligned + distinguish = Speh}. We do not have a conceptual explanation for this now. The following observation might be helpful when applying this lemma. For a pair $(H_{p,q},\mu_a)$, set $S^+ (p,q,a) = p - q + 2a$ and $S^-(p,q,a) = p - q - 2a$. When passing from distingusihed $\sigma_1 \times \sigma_2$ to distingsuihed $\rho_1$ and $\rho_4$, the invariants $S^+$ and $S^-$ for the subgroup pairs are preserved respectively.
\end{rem}

To handle the duality relation in \eqref{formula::necessary condition--GL}, we have the following

\begin{lem}\label{lem::duality--multisegments}
	Let $\mathfrak{m}_1,\cdots,\mathfrak{m}_r$ and $\mathfrak{n}_1,\cdots,\mathfrak{n}_s$ be multisegments. If
	\begin{align*}
		 \langlands(\mathfrak{m}_1) \times \cdots \times \langlands(\mathfrak{m}_r) \cong \langlands(\mathfrak{n}_1) \times \cdots \times \langlands(\mathfrak{n}_s),
	\end{align*}
    then $\mathfrak{m}_1 + \cdots + \mathfrak{m}_r  = \mathfrak{n}_1 + \cdots + \mathfrak{n}_s$.
\end{lem}
\begin{proof}
	It is known that $\langlands(\mathfrak{m}_1 + \cdots + \mathfrak{m}_r)$ is a subquotient of $\langlands(\mathfrak{m}_1) \times \cdots \times \langlands(\mathfrak{m}_r)$. By our condition, it is then a subquotient of $\lambdaup(\mathfrak{n}_1 + \cdots + \mathfrak{n}_s)$. Reversing the roles of $\mathfrak{m}_i$'s and $\mathfrak{n}_j$'s, the required equality follows from \cite[Theorem 7.1]{Zelevinsky-II} (see also \cite[Theorem 5.3]{Tadic-Induced-GLn-DivisionAlgebra}).
\end{proof}

As seen from above, the geometric lemma provides us necessary conditions for distinction of induced representations.We now present a sufficient condition that is due to Matringe \cite[Proposition 3.8]{Matringe-BF+L-function}.
\begin{lem}\label{lem::preservation-parabolic induction}
	Let $n_1 = 2 m_1$ and $n_2 = 2 m_2$ be even integers, let $a \in \BR$. Assume that $\pi_1$ is $(H_{m_1,m_1},\mu_a )$-distinguished and $\pi_2$ is $(H_{m_2,m_2} , \mu_a)$-distinguished. Then $\pi_1 \times \pi_2$ is $(H_{m_1+m_2,m_1+m_2},\mu_a)$-distinguished.
\end{lem}

\subsection{Distinction of products of essentially square-integrable representations}

We now apply Corollary \ref{cor::Distinction--maximal parabolic} to products of essentially square-integrable representations.

\begin{prop}\label{prop::Geometric Lemma-Main}
	Let $\pi = \Delta_1 \times \cdots \times \Delta_t$ be a representation of $G_n$, where $\Delta_i = \Delta([a_i,b_i]_{\rho_i})$ is an essentially square-integrable representation of $G_{n_i}$, $i = 1, \cdots ,t$. (Here we assume all $a_i$, $b_i$ are integers.) Suppose that $\pi$ is $(H_{p,q},\mu_a)$-distinguished with $p$, $q$ two nonnegative integers, $p + q =n$, and $a \in \BR$. Then there exist an integer $c_t$ satisfying $a_t - 1 \leqslant c_t  \leqslant b_t$ such that one of the following cases must hold: 
	
	\textup{Case A1.}$\ $ One has $a_t = c_t < b_t$. The representation $\Delta([a_t,c_t]_{\rho_t}) = \mathbf{b}(\Delta_t) $ is either the character $\nu^{a+(q-p+1)/2}$ or the character $\nu^{-a + (p-q + 1)/2}$ of $G_1$; and there exists $i \in \{ 1, 2, \cdots, t-1\}$, and an integer $c_{i}$, $a_{i} \leqslant c_{i} \leqslant b_i$, such that
	
	\begin{itemize}
		\item[(i)] one has $\Delta([a_t+1,b_t]_{\rho_t})^{\vee}  \cong \Delta([a_i,c_i]_{\rho_i})$;
		\item[(ii)] the representation
		\begin{align*}
			\Delta_1 \times \cdots \times \Delta([c_i + 1,b_i]_{\rho_i}) \times \cdots \times \Delta_{t-1}
		\end{align*}
		is $(H_{p-n_t, 1 + q -n_t},\mu_{a + 1/2})$ or $(H_{ 1 +p-n_t, q -n_t},\mu_{a - 1/2})$-distinguished, depending on $\mathbf{b}(\Delta_t)$.
	\end{itemize}
	
	\textup{Case A2.}$\ $ One has $a_t \leqslant c_t < b_t$. The representation $\Delta([a_t,c_t]_{\rho_t})$, with its degree $n_t'$ an even integer, is $H_{n_t'/2,n_t'/2}$-distinguished; and there exists $i \in \{1, 2,\cdots, t-1\}$ and an integer $c_{i}$, $a_{i} \leqslant c_{i} \leqslant b_i$, such that
	
	\begin{itemize}
		\item[(i)] one has $\Delta([c_t+1,b_t]_{\rho_t})^{\vee}  \cong \Delta([a_i,c_i]_{\rho_i})$;
		\item[(ii)] the representation
		\begin{align*}
			\Delta_1 \times \cdots \times \Delta([c_i + 1,b_i]_{\rho_i}) \times \cdots \times \Delta_{t-1}
		\end{align*}
		is $(H_{p', q'},\mu_{a})$ -distinguished with $p' = p-n_t + n_t'/2$ and $q' = q - n_t + n_t'/2$.
	\end{itemize}
	
	\textup{Case B1.}$\ $ One has $c_t = b_t$. The representation $\Delta([a_t,c_t]_{\rho_t}) = \Delta_t$ is either the character $\nu^{a+(q-p+1)/2}$ or the character $\nu^{-a + (p-q + 1)/2}$ of $G_1$; and the representation
	\begin{align*}
		\Delta_1 \times \cdots \times \Delta_{t-1}
	\end{align*}
	is $(H_{p-1, q},\mu_{a + 1/2})$ or $(H_{ p, q - 1},\mu_{a - 1/2})$-distinguished, depending on $\Delta_t$.
	
	\textup{Case B2.}$\ $ One has $c_t = b_t$. The representation $\Delta_t$ is $H_{n_t/2,n_t/2}$-distinguished, where $n_t$ is even; and the representation
	\begin{align*}
		\Delta_1 \times \cdots \times \Delta_{t-1}
	\end{align*}
	is $(H_{p-n_t/2, q-n_t/2},\mu_a)$-distinguished.
	
	\textup{Case C.}$\ $ One has $c_t = a_t - 1$. There exists $i \in \{1,2,\cdots, t-1 \}$ and an integer $c_{i}$, $a_{i} \leqslant c_{i} \leqslant b_i$, such that
	
	\begin{itemize}
		\item[(i)] one has $\Delta_t^{\vee}  \cong \Delta([a_i,c_i]_{\rho_i})$;
		\item[(ii)] the representation
		\begin{align*}
			\Delta_1 \times \cdots \times \Delta([c_i + 1,b_i]_{\rho_i}) \times \cdots \times \Delta_{t-1}
		\end{align*}
		is $(H_{p-n_t, q -n_t},\mu_a)$-distinguished.
	\end{itemize}
\end{prop}

\begin{proof}
	Write $\sigma_1 = \Delta_1 \times \cdots \times \Delta_{t-1}$ and $\sigma_2 = \Delta_t$, and $k = n - n_t$. By Corollary \ref{cor::Distinction--maximal parabolic}, in its notation, there exists $(r,s) \in \mathcal{I}^k_{p,q}$ with defect $d = k - r - s$ such that the representation $r_{(k-d,d)} \sigma_1 \otimes r_{(d,n-k-d)}\sigma_2$ of $L$ is $(L_x, \delta_{Q_x}\delta_Q^{-1/2}\mu_a^{\eta})$-distinguished. By \cite[9.5]{Zelevinsky-II}, the Jacquet module $r_{(d,n-k-d)} \sigma_2$ of $\sigma_2$ is either zero or of the form $\Delta([c_t + 1, b_t]_{\rho_t}) \otimes \Delta([a_t,c_t]_{\rho_t})$ for certain integer $c_t$ with $ a_t - 1 \leqslant c_t \leqslant b_t$. By \cite[1.2, 1.6]{Zelevinsky-II}, there exists a filtration $0 \subset V_1 \subset \cdots \subset V= r_{(k-d,d)} \sigma_1$ such that each successive factor is equivalent to a representation of the form
	\begin{align*}
		\Delta([c_1 + 1,b_1]_{\rho_1}) \times \cdots \times \Delta([c_{t-1} + 1,b_{t-1}]_{\rho_{t-1}})  \otimes \Delta([a_1,c_1]_{\rho_1}) \times \cdots  \times \Delta([a_{t-1},c_{t-1}]_{\rho_{t-1}}),
	\end{align*}
	for certain integers $c_i$ such that $a_i - 1 \leqslant c_i \leqslant b_i$, $i = 1, \cdots,t-1$. Therefore, there exists integers $c_i$, $i =1 ,2,\cdots ,t$, such that the pure tensor representation
	\begin{align*}
		\prod_{i=1}^{t-1} \Delta([c_i +1, b_i]_{\rho_i}) \otimes \prod_{i=1}^{t-1} \Delta([a_i,c_i]_{\rho_i}) \otimes \Delta([c_t + 1, b_t]_{\rho_t}) \otimes \Delta([a_t,c_t]_{\rho_t})
	\end{align*}
	is $(L_x,\delta_{Q_x}\delta_Q^{-1/2}\mu_a^{\eta^{-1}})$-distinguished. By Lemma \ref{lem::distinction--pure tensor}, we have 
	\begin{align*}
		\Delta([c_t +1, b_t]_{\rho_t})^{\vee}  \cong \prod_{i=1}^{t-1} \Delta([a_i,c_i]_{\rho_i}).
	\end{align*}
	By Lemma \ref{lem::duality--multisegments}, $c_i = a_i - 1$ for all but one $i$ between $1$ and $t-1$. So, for this $i$, we have
	\begin{align}\label{formula::geometric lemma--Main--Mid}
		\Delta([c_t +1, b_t]_{\rho_t})^{\vee}  \cong \Delta([a_i , c_i]_{\rho_i}).
	\end{align}
	Lemma \ref{lem::distinction--pure tensor} also implies that 
	\begin{align}\label{formula::geometric lemma--Main--4}
		\text{  $\Delta([a_t,c_t]_{\rho_t})$ is $(H_{p+s-k,q+r-k}, \mu_{a + (s-r)/2})$-distinguished,}  
	\end{align}
	and that
	\begin{align}\label{formula::geometric lemma--Main--1}
		\Delta_1 \times \cdots \times \Delta([c_i + 1,b_i]_{\rho_i}) \times \cdots \times \Delta_{t-1} \nonumber\\
		\text{is $(H_{r,s}, \mu_{a + (p-q + s - r) /2})$-distinguished.}  
	\end{align}
	%	is $(G_r \times G_s, \mu_{a + (p-q + s - r) /2})$-distinguished. 
	
	When $a_t \leqslant c_t < b_t$, we have two subcases. If $c_t = a_t$ and the degree of $\rho_t$ equals to $1$, it follows from \eqref{formula::geometric lemma--Main--4} that $(p+s-k, q+r-k) = (1,0)$ or $(0,1)$. By \eqref{formula::geometric lemma--Main--1}, \eqref{formula::geometric lemma--Main--Mid} and simple calculations, we then have Case A1; Otherwise, the representation $\Delta([a_t,c_t]_{\rho_t})$ is not one dimensional. Thus, in \eqref{formula::geometric lemma--Main--4} we have $p+s-k >0$ and $q+r-k >0$. By Proposition \ref{prop::starting point}, we get that $\Delta([a_t,c_t]_{\rho_t})$ is $H_{n_t'/2,n_t'/2}$-distinguished with $n_t'$ its degree. The rest statements of Case A2 follow from simple calculations. Thus we have Case A2.
	
	When $c_t = b_t$, we have two subcases. If $\Delta_t$ is a character of $G_1$, then by similar arguments as in Case A1, we have Case B1. Otherwise, by similar arguments as in Case A2, we have Case B2. In these two cases, we have $d = 0$ and $c_i = a_i - 1$ by our convention.
	
	When $c_t = a_t - 1$, by \eqref{formula::geometric lemma--Main--4}, we have $p+s-k = q + r -k = 0$. The statements of Case C follow from \eqref{formula::geometric lemma--Main--1}, \eqref{formula::geometric lemma--Main--Mid} and simple calculations. So we are done.
\end{proof}

\begin{cor}\label{cor::Geometric Lemma-Main}
	Let $\pi = \Delta_1 \times \cdots \times \Delta_t$ be as above. If $\pi$ is $(H_{p,q},\mu_a)$-distinguished with $p$, $q$ and $a$ as above, then either the representation $\Delta_t$ is the character $\nu^{a+(q-p+1)/2}$ or the character $\nu^{-a + (p-q + 1)/2}$ of $G_1$, or there is $i \in \{ 1, 2 , \cdots, t\}$ such that $\mathbf{e}(\Delta_t)^{\vee} \cong \mathbf{b}(\Delta_i)$. 
\end{cor}
\begin{proof}
	Note that in all cases other than Case B1, we have a duality relation.
\end{proof}

Considering the duality relation between extremities of segments, a generalization of Corollary \ref{cor::Geometric Lemma-Main} is given later in Proposition \ref{prop::Geometric Lemma--Products Products}.

%Considering the duality relation between extremities of segments, we may generalize Corollary \ref{cor::Geometric Lemma-Main} slightly. The proof of the %following proposition is similar to that of Proposition \ref{prop::Geometric Lemma-Main} and is omitted here.
%
%\begin{prop}\label{prop::Geometric Lemma--Products Products}
%	Let $\pi = \Delta_1 \times \cdots \times \Delta_t$ and $\pi' = \Delta'_1 \times \cdots \times \Delta'_s$ be two produts of essentially square-integrable %representations. If $\pi \times \pi'$ is $(H_{p,q},\mu_a)$-distinguished for two nonnegative integers $p$, $q$, $p + q =n $ and $a \in \BR$, then there are two %possibilities here:
%	
%	(1).$\ $One has $\pi$ is $(H_{p_1,q_1},\mu_{a_1})$-distinguished and  $\pi'$ is $(H_{p_2,q_2},\mu_{a_2})$-distinguished for some $p_i$, $q_i$ and $a_i$, $i %= 1, 2$. Here the subscripts $(p_i,q_i,a_i)$, $i = 1,2$, satisfy
%	\begin{align*}
%		\begin{cases*}
%			p_1 + p_2 = p \\
%			q_1 + q_2 = q
%		\end{cases*}
%		and \,
%		\begin{cases*}
%			p_1 - q_1 + 2 a_1 = p - q + 2 a \\
%			p_2 - q_2 - 2 a_2 = p - q - 2a.
%		\end{cases*}
%	\end{align*}
%	(2).$\ $There exist $i \in \{1,\cdots,t\}$ and $j \in \{ 1, \cdots ,s\}$ such that $\mathbf{e}(\Delta'_j)^{\vee} \cong \mathbf{b}(\Delta_i)$.
%\end{prop}

\section{Distinction of Ladder Representations}\label{section::ladder}

\subsection{Notations and basic facts}

The class of ladder representations was first introduced by Lapid and M\'{i}nguez in \cite{Lapid-Minguez--Ladder}, and was further studied by Lapid and his collaborators in \cite{Lapid-Kret-JacquetModule+Ladder} and \cite{Lapid-Minguez-ParabolicInduction}. We start by reviewing some basic facts of these representations.

\subsubsection{Definitions}

Let $\rho \in \mathscr C$. By a \emph{ladder} we mean a set $\{\Delta_1,\cdots, \Delta_t  \}  \in \mathscr O_{\rho}$ such that
\begin{align}\label{formula::ladder-definition}
	\mathbf{b}(\Delta_1) > \cdots > \mathbf{b}(\Delta_t) \quad \text{and}\quad \mathbf{e}(\Delta_1) > \cdots > \mathbf{e}(\Delta_t).
\end{align}
A representation $\pi \in \textup{Irr}$ is called a \emph{ladder representation} if $\pi = \langlands(\mathfrak{m})$ where $\mathfrak{m} \in \mathscr O_{\rho}$ is a ladder. Whenever we say that $\mathfrak{m} = \{\Delta_1,\cdots ,\Delta_t \} \in \mathscr O_{\rho}$ is a ladder, we implicitly assume that $\mathfrak{m}$ is already ordered as in \eqref{formula::ladder-definition}. We denote by $\mathfrak{m}^{\vee} \in \mathscr O_{\rho^{\vee}}$ the ladder $\{\Delta_t^{\vee},\cdots ,\Delta_1^{\vee}\}$. 

\begin{lem}\label{lem::DualOfLadder}
	Let $\mathfrak{m} \in \mathscr O_{\rho}$ be a ladder. One has $\langlands(\mathfrak{m})^{\vee}  =  \langlands(\mathfrak{m}^{\vee})$.
	% 	\begin{align}
	% 	  ( \langlands( \Delta([a_1,b_1]_{\rho} \times \cdots \times \Delta([a_t,b_t]_{\rho})) )^{\vee} \\
	% 	  \cong \langlands( \Delta([-b_t,-a_t]_{\rho^{\vee}} \times \cdots \times \Delta( [-b_1,-a_1]_{\rho^{\vee}})  )    ).
	% 	\end{align}
\end{lem}
\begin{proof}
	See \cite[Proposition 5.6]{Tadic-86-UnitaryDual-GLn}
\end{proof}
We introduce some more notation. For a ladder $\mathfrak{m} = \{ \Delta_1,\cdots, \Delta_t \} \in \mathscr O_{\rho}$ ordered as in \eqref{formula::ladder-definition}, set $\pi = \langlands(\frm)$. We shall denote $\mathbf{b}(\Delta_1)$ by $\mathbf{b}(\pi)$, called the beginning of the ladder representaion $\pi$; denote $\mathbf{e}(\Delta_t)$ by $\mathbf{e}(\pi)$, called the end of $\pi$. We shall denote the number $t$ of segments in $\mathfrak{m}$ by $\mathbf{ht}(\pi)$, called the height of $\pi$. 

We say that $\pi$ is a \emph{decreasing} (resp. \emph{increasing}) ladder representation if
\begin{align*}
	l(\Delta_1) \geqslant \cdots \geqslant l(\Delta_t)  \quad (\textup{resp. }l(\Delta_1) \leqslant \cdots \leqslant l(\Delta_t) ).
\end{align*}
We say that $\pi$ is a \emph{left aligned} (resp. \emph{right aligned}) representation if $\mathbf{b}(\Delta_i)  =  \mathbf{b}(\Delta_{i+1}) + 1$ (resp. $\mathbf{e}(\Delta_i)  =  \mathbf{e}(\Delta_{i+1}) + 1$ ), $i= 1,\cdots, t-1$. Note that left aligned repreesentations are decreasing ladder representations and right aligned representations are increasing ladder representations.  

A ladder representation is called an \emph{essentially Speh} representation if it is both left aligned and right aligned. Note that essentially Speh representations are just the usual Speh representations up to twist by a non-unitary character. Let $\Delta$ be an essentially square-integrable representation of $G_d$ and $k$ a positive integer. Then $\frm_1 = \{\nu^{(k-1)/2}\Delta, \nu^{(k-3)/2}\Delta, \cdots, \nu^{(1-k)/2}\Delta \}$ is a ladder, and the ladder representation $\langlands(\frm_1)$ is an essentially Speh representation, which we denote by $\speh(\Delta,k)$. All essentially Speh representations can be obtained in this manner.

%%%%---\footnote{We caution the read that the notion `essentially Speh representations' has a different meaning in the work \cite{Chan-KeiYuen--Proof+GGP+NonTempered}, %%%%---where they are defined to be the non-zero derivatives of Speh representations, the same class of representations called `quasi-Speh' in %%%%---\cite{Gurevich--Proof+GGP+NonTempered--JEMS}.}

Let $\pi = \langlands(\mathfrak{m})$ as above. Let us further write $\Delta_i  =  \Delta([a_i,b_i]_{\rho})$. (The $a_i$'s are integers by our convention.) By a \emph{division} of $\pi$ as two ladder representations $\pi'$ and $\pi''$, denoted by $\pi = \pi' \sqcup \pi''$, we mean that there exist integers $c_i$ with $a_i - 1 \leqslant c_i \leqslant b_i$, $i = 1, \cdots, t$, such that 
\begin{align*}
	c_1 > c_2 > \cdots > c_t
\end{align*}
and that
\begin{align*}
	\pi'  &=  \langlands(\Delta([a_1,c_1]_{\rho}), \cdots, \Delta([a_t,c_t]_{\rho})),  \\
	\pi'' &=  \langlands(\Delta([c_1 + 1, b_1]_{\rho}), \cdots, \Delta([c_t + 1, b_t]_{\rho})).
\end{align*}
Note that if $\pi$ is an essentially Speh representation and $\pi = \pi' \sqcup \pi''$ with neither $\pi'$ nor $\pi''$ the trivial representation of $G_0$, then we have $\mathbf{b}(\pi) = \mathbf{b}(\pi')$ and $\mathbf{e}(\pi)  =  \mathbf{e}(\pi'')$.
\subsubsection{Standard module}

One useful property of ladder representations is that the relation between them and their standard modules is explicit. Let $\frm = \{\Delta_1,\cdots, \Delta_t  \} \in \mathscr O_{\rho}$ be a ladder with $\Delta_i  =  \Delta([a_i, b_i])_{\rho}$. Set
\begin{align*}
	\mathcal{K}_i = \Delta_1 \times \cdots \times \Delta_{i-1} \times \Delta([a_{i+1}, b_i]_{\rho}) \times \Delta([a_i,b_{i+1}]_{\rho})  \times \Delta_{i+1} \times \cdots \times \Delta_t,
\end{align*}
for $i = 1, \cdots, t-1$. (By our convention, $\mathcal{K}_i  =  0$ if $a_i > b_{i+1} + 1$). By \cite[Theorem 1]{Lapid-Minguez--Ladder} we have
\begin{prop}\label{prop::ladder--kernel description}
	With the above notation let $\mathcal{K}$ be the kernel of the projection $\lambda(\frm) \ra \textup{L}(\frm)$. Then $\mathcal{K} = \sum_{i=1}^{t-1} \mathcal{K}_i$.
\end{prop}

\subsubsection{Jacquet modules}

The Jacquet modules of ladder representations were computed in \cite[Corollary 2.2]{Lapid-Kret-JacquetModule+Ladder}, where it is shown that the Jacquet module of a ladder representation is semisimple, multiplicity free, and that its irreducible constituents are themselves tensor products of ladder representations. For us, we need only the Jacquet modules with respect to maximal parabolic subgroups. We record the result in \cite{Lapid-Kret-JacquetModule+Ladder} here. Let $P = M \ltimes U$ be the standard parabolic subgroup of $G_n$ associated to $(k,n-k)$.

\begin{prop}\label{prop::ladder--Jacquet Module}
	Let $\frm = \{\Delta_1,\cdots, \Delta_t  \} \in \mathscr O_{\rho}$ be a ladder with $\Delta_i  =  [a_i, b_i]_{\rho}$, and $\pi = \langlands(\frm)$. Then
	\begin{align*}
		r_{M,G} (\pi) = \sideset{}{'}\sum_{\pi = \pi_1 \sqcup \pi_2} \pi_2 \otimes \pi_1,
	\end{align*}
	where the summation takes over all divisions of $\pi$ as two ladder representations $\pi_1$ and $\pi_2$ such that the degree of $\pi_1$ is $ n - k$ and that the degree of $\pi_2$ is $k$.
\end{prop}

\subsubsection{Bernstein-Zelevinsky derivatives} The full derivative of a ladder representation was computed in \cite[Theorem 14]{Lapid-Minguez--Ladder}, where it is shown that the semisimplification of all of the derivatives of a ladder representation consists of ladder representations of smaller groups. In particular, the derivatives of a left aligned representation take simple forms, which we recall here.

\begin{lem}\label{lem::ladder-derivative-left aligned}
	Let $\rho \in \mathscr C (G_d)$, and $\frm = \{\Delta_1,\cdots, \Delta_t  \} \in \mathscr O_{\rho}$ be a ladder with $\Delta_i  =  \Delta([a_i, b_i]_{\rho})$. Suppose that $\pi = \langlands(\frm)$ is a left aligned representation. If $k$ is not divided by $d$, then $\pi^{(k)}  =  0$. If $k = rd$, then 
	\begin{align*}
		\pi^{(k)}  =  \langlands(\Delta([a_1 + r,b_1]_{\rho}),\Delta_2, \cdots ,\Delta_t).
	\end{align*}
\end{lem}

\subsection{Distinction of products of essentially Speh representations}

In this subsection we apply Corollary \ref{cor::Distinction--maximal parabolic} to products of essentially Speh representations. 

Instead of Lemma \ref{lem::duality--multisegments}, we will use the following lemma to handle the duality relation in consequences of the geometric lemma.

\begin{lem}\label{lem::ladder-identification left aligned}
	Let $\sigma$ and $\pi_i$ be left aligned representations of $G_n$ and $G_{n_i}$, $i = 1, \cdots, k$. If $\sigma \cong \pi_1 \times \cdots \times \pi_k$, then   $k = 1$.
\end{lem}
\begin{proof}
	By Lemma \ref{lem::ladder-derivative-left aligned}, the derivatives of left aligned representations are either $0$ or irreducible representations. Our assertion then follows form the description of the derivatives of a product of representations in \cite[Corollary 4.6]{BZ-I}  
\end{proof}

In view of Lemma \ref{lem::ladder-identification left aligned} and the description of Jacquet modules of a ladder representation in Proposition \ref{prop::ladder--Jacquet Module}, we formulate the following proposition, whose proof is very similar to that of Proposition \ref{prop::Geometric Lemma-Main} and is omitted here.

\begin{prop}\label{prop::ladder-products of left aligned}
	Let $\pi = \pi_1 \times \cdots \times \pi_t $ be a representation of $G_n$, where $\pi_i$ is an essentiallly Speh representation of $G_{n_i}$, $i = 1 , \cdots, t$. Assume that $\pi$ is $(H_{p,q},\mu_a)$-distinguished with $p$, $q$ two nonnegative integers, $p + q =n$ and $a \in \BR$. Then there exist a division of $\pi_t$ as two ladder representations $\pi_t'$ and $\pi_t''$, $\pi_t  =  \pi_t'  \sqcup \pi_t''$, with degrees $n_t'$ and $n_t''$ respectively, such that one of the following cases must hold: 
	
	\textup{Case A}.$\ $ The representation $\pi_t'$ is neither $\pi_t$ nor the trivial representation of $G_0$. There exists $i_0$, $1 \leqslant i_0 \leqslant t-1$, and a division of $\pi_{i_0}$ as two ladder representations $\pi_{i_0}'$ and $\pi_{i_0}''$, $\pi_{i_0}  =  \pi_{i_0}'  \sqcup \pi_{i_0}''$, such that
	\begin{itemize}
		\item[(i)]  $\pi_t'$ is $(H_{r,s},\mu_{a+(r-s+q-p)/2})$-distinguished, for two nonnegative integers $r, s \geqslant 0$, $ r + s = n_t'$;
		\item[(ii)]  $\pi_t''^{\vee} \cong \pi_{i_0}'$;
		\item[(iii)]  the representation 
		\begin{align}\label{formula::remaining repns-I}
			\pi_1 \times \cdots \times \pi_{i_0 -1} \times \pi_{i_0}'' \times \pi_{i_0 +1} \times \cdots \times \pi_{t-1}
		\end{align}
		is $(H_{r',s'},\mu_{a+(s'-r'+ p-q)/2})$-distinguished, for two nonnegative integers $r'$, $s' \geqslant 0$, $r' + s' = n - n_t - n_t''$.
	\end{itemize}
	
	\textup{Case B}.$\ $ One has $\pi_t' = \pi_t$ is $(H_{r,s},\mu_{a+(r-s+q-p)/2})$-distinguished, for two nonnegative integers $r$, $s \geqslant 0$, $ r + s = n_t$, and the representation 
	\begin{align}\label{formula::remaining repns--II}
		\pi_1 \times \cdots \times \pi_{t-1}
	\end{align}
	is $(H_{r',s'},\mu_{a+(s'-r'+ p-q)/2})$-distinguished, for two nonnegative integers $r',s' \geqslant 0$, $r' + s' = n - n_t$.
	
	\textup{Case C}.$\ $The representation $\pi_t'$ is the trivial representation of $G_0$, so $\pi_t'' = \pi_t$. There exists $i_0$, $1 \leqslant i_0 \leqslant t-1$, and a division of $\pi_{i_0}$ as two ladder representations $\pi_{i_0}'$ and $\pi_{i_0}''$, $\pi_{i_0}  =  \pi_{i_0}'  \sqcup \pi_{i_0}''$, such that
	\begin{itemize}
		\item[(i)]  $\pi_t^{\vee} \cong \pi_{i_0}'$;
		\item[(ii)]  the representation 
		\begin{align}\label{formula::remaining repns--III}
			\pi_1 \times \cdots \times \pi_{i_0 -1} \times \pi_{i_0}'' \times \pi_{i_0 +1} \times \cdots \times \pi_{t-1}
		\end{align}
		is $(H_{r,s},\mu_{a+(s-r+ p-q)/2})$-distinguished, for two nonnegative integers $r,s \geqslant 0$, $r + s = n - 2 n_t $.
	\end{itemize}
\end{prop}

\begin{rem}
	It is easy to see that Lemma \ref{lem::ladder-identification left aligned} fails if one removes the condition that $\sigma$ is left aligned. This is the reason why we restrict ourselves to products of essentially Speh representations here. Proposition \ref{prop::ladder-products of left aligned} makes an inductive proof of the classification result possible as, in many cases, the representations  \eqref{formula::remaining repns-I}, \eqref{formula::remaining repns--II} and \eqref{formula::remaining repns--III} are still products of essentially Speh representations.
\end{rem}

Nevertheless, we have the following proposition for products of ladder representations that is very useful in later arguments.

\begin{prop}\label{prop::Geometric Lemma--Products Products}
	Let $\Pi = \pi_1 \times \cdots \times \pi_t$ and $\Pi' = \pi'_1 \times \cdots \times \pi'_s$ be two products of ladder representations. If $\Pi \times \Pi'$ is $(H_{p,q},\mu_a)$-distinguished for two nonnegative integers $p$, $q$, $p + q =n $ and $a \in \BR$, then there are two possibilities here:
	
	(1).$\ $ $\Pi$ is $(H_{p_1,q_1},\mu_{a_1})$-distinguished and  $\Pi'$ is $(H_{p_2,q_2},\mu_{a_2})$-distinguished for some $p_i$, $q_i$ and $a_i$, $i = 1, 2$. Here the subscripts $(p_i,q_i,a_i)$, $i = 1,2$, satisfy
	\begin{align*}
		\begin{cases*}
			p_1 + p_2 = p \\
			q_1 + q_2 = q
		\end{cases*}
		and \,
		\begin{cases*}
			p_1 - q_1 + 2 a_1 = p - q + 2 a \\
			p_2 - q_2 - 2 a_2 = p - q - 2a.
		\end{cases*}
	\end{align*}
	
	(2).$\ $There exist $i \in \{1,\cdots,t\}$ and $j \in \{ 1, \cdots ,s\}$ such that $\mathbf{e}(\pi'_j)^{\vee} \cong \mathbf{b}(\pi_i)$.
\end{prop}
\begin{proof}
	This follows from similar arguments of Proposition \ref{prop::Geometric Lemma-Main} and the following simple implication of  Lemma \ref{lem::duality--multisegments} when applied to ladder representations. 
\end{proof}

\begin{lem}
		Let $\mathfrak{m}_1,\cdots,\mathfrak{m}_r$ and $\mathfrak{n}_1,\cdots,\mathfrak{n}_s$ be ladders. If
	\begin{align*}
		\langlands(\mathfrak{m}_1) \times \cdots \times \langlands(\mathfrak{m}_r) \cong \langlands(\mathfrak{n}_1) \times \cdots \times \langlands(\mathfrak{n}_s),
	\end{align*}
       then there exist  $i \in \{1,\cdots,r\}$ and $j \in \{ 1, \cdots ,s\}$ such that $\mathbf{b}(\langlands(\mathfrak{m}_i)) \cong \mathbf{b}(\langlands(\mathfrak{n}_j))$.
\end{lem}
\begin{proof}
	By Lemma \ref{lem::duality--multisegments}, one has $\mathfrak{m}_1+ \cdots + \mathfrak{m}_r = \mathfrak{n}_1 + \cdots + \mathfrak{n}_s$. Write $\mathfrak{m}_i = \{\Delta_{i,1},\cdots,\Delta_{i,k_i} \}$ and $\mathfrak{n}_j = \{ \Delta'_{j,1},\cdots, \Delta'_{j,l_j}\}$ for these $i$'s and $j$'s. Let $\Delta$ be one segement in $\sum_i \mathfrak{m}_i$ such that $\mathbf{b}(\Delta)$ is maximal, which means that, if for some $\Delta_0 \in \sum_i \mathfrak{m}_i$ with $\mathbf{b}(\Delta_0)$ lying in the same cuspidal line with $\mathbf{b}(\Delta)$, then $\mathbf{b}(\Delta_0) \leqslant \mathbf{b}(\Delta)$. As these $\mathfrak{m}_i$'s are ladders, one has $\Delta \in \{\Delta_{1,1},\cdots,\Delta_{r,1} \}$. Also, one has $\Delta \in \{\Delta'_{1,1}, \cdots, \Delta'_{s,1} \}$. So the lemma follows.
\end{proof}

It will turns out that the ordering of representations in a product is important for the geometric lemma approach to distinction problems. The commutativity of a product of two ladder representations was studied by Lapid and M\'{i}nguez in \cite{Lapid-Minguez-ParabolicInduction}. Here we present a special case of their results that is sufficient for our purpose.

\begin{lem}\label{lem::ladder--Commutativity}
	Let $\rho \in \mathscr C$. Let $\frm_1, \frm_2 \in \mathscr O_{\rho}$ be two ladders, with $\frm_1  =  \{\Delta_{1,1},\cdots, \Delta_{1,\shskip t_1} \}$ and $\frm_2  =  \{\Delta_{2,1},\cdots, \Delta_{2,\shskip t_2} \}$. Suppose that $\langlands(\frm_1)$ is an essentially Speh representation and $\langlands(\frm_2)$ is a right aligned representation. If $\mathbf{e}(\Delta_{1,\shskip t_1})  =  \mathbf{e}(\Delta_{2,\shskip t_2})$ and $t_2 \leqslant t_1$, or $\mathbf{e}(\Delta_{1,\shskip t_1})  =  \mathbf{e}(\Delta_{2,\shskip t_2})$ and $\mathbf{b}(\Delta_{1,\shskip t_1})  \leqslant  \mathbf{b}(\Delta_{2,\shskip t_2})$, then $\langlands(\frm_1) \times \langlands(\frm_2)$ is irreducible and $\langlands(\frm_1) \times \langlands(\frm_2)   =   \langlands(\frm_2) \times \langlands(\frm_1)$.
\end{lem}
\begin{proof}
	Note that the results in \cite{Lapid-Minguez-ParabolicInduction} are expressed in terms of Zelevinsky classification. By the combinatorial description of Zelevinsky involution by Moeglin-Waldspurger \cite{Moeglin-Waldspurger-ZelevinskyInvolution} (see also \cite[\S 3.2]{Lapid-Minguez--Ladder}), we can rewrite the conditions in the lemma in terms of the Zelevinsky involution $\frm_1^t$ and $\frm_2^t$ of $\frm_1$ and $\frm_2$. The assertion then follows from Proposition 6.20 and Lemma 6.21 in \cite{Lapid-Minguez-ParabolicInduction}.
\end{proof}

%Till now, our The aim of this section is to prove Theorem \ref{thm::Speh + distinguishe = max Levi G_m,m} that generalizes Corollary \ref{cor::discrete %series--max Levi} to the case of essentially Speh representations. A key ingredient in its proof is Proposition \ref{prop::Speh--Nece+MaxiLevi} which also %gives a proof of the `only if' part of Theorem \ref{thm::1 Speh}.

\subsection{Distinction of essentially Speh representations}\label{section::Speh--MaxLevi}

From now on, we shall perform some detailed analysis using our consequences of the geometric lemma. 

\begin{prop}\label{prop::ladder-Speh--selfdual}
	Let $\pi$ be an essentially Speh representation of $G_n$. If $\pi$ is $(H_{p,q},\mu_a)$-distinguished for two positive integers $p$, $q$ with $p + q = n$ and some $a \in \BR$, then $\pi$ is self-dual.
\end{prop}
\begin{proof}
    Write $\pi = \langlands (\frm)$ with $\frm = \{\Delta_1,\cdots,\Delta_t  \}$ a ladder. The case where $\pi$ is one dimensional is obvious. So we assume that $\pi$, hence $\Delta_t$, is not one dimensional. By assumption, $\Delta_1 \times \cdots \times \Delta_t$ is $(H_{p,q},\mu_a)$-distinguished. By Corollary \ref{cor::Geometric Lemma-Main}, there exists $i$, $1 \leqslant i \leqslant t$, such that 
    \begin{align}\label{formula::ess-speh-1}
    	\mathbf{b}(\Delta_i) \cong \mathbf{e}(\Delta_t)^{\vee}.
    \end{align}
    We claim that $i = 1$. If so, by Lemma \ref{lem::DualOfLadder}, we see that $\pi$ is self-dual. In fact, if otherwise $i > 1$, we apply Proposition \ref{prop::Geometric Lemma--Products Products} to $\pi_1 \times \pi_2$, where $\pi_1 = \Delta_1 \times \cdots \times \Delta_{i-1}$ and $\pi_2 = \Delta_i \times \cdots \times \Delta_t$. We get either that 
    \begin{align}\label{formula::ess-speh-2}
    	 \mathbf{b}(\Delta_j) \cong \mathbf{e}(\Delta_k)^{\vee}
    \end{align}
    for some $j$, $ 1 \leqslant j \leqslant i-1$ and some $k$, $ i \leqslant k \leqslant t$, or that $\pi_1$ is $(H_{p_1,q_1},\mu_{a_1})$-distinguished with some $p_1,q_1$ and $a_1$, which implies, using Corollary \ref{cor::Geometric Lemma-Main} again, that
    \begin{align}\label{formula::ess-speh-3}
    	\mathbf{b}(\Delta_l) \cong \mathbf{e}(\Delta_{i-1})^{\vee}
    \end{align}
    for some $l, 1\leqslant l \leqslant i-1$. But we see easily that both \eqref{formula::ess-speh-2} and \eqref{formula::ess-speh-3} contradict with \eqref{formula::ess-speh-1}.
\end{proof}

\begin{cor}\label{cor::Speh--Nece+MaxiLevi--CharacterInGeneralPosition}
	Let $\pi$ be an essentially Speh representation of $G_n$. If the representation $\pi$ is $(H_{p,q},\mu_a)$-distinguished for two positive integers $p$, $q$ with $p + q = n$ and some $a \in \BR$, $a \neq 0$, then $p = q$.
\end{cor}
\begin{proof}
	This follows from Proposition \ref{prop::ladder-Speh--selfdual} and consideration of the central character of $\pi$.
\end{proof}

Now we are in a position to prove one direction of Theorem \ref{thm::1 Speh} (what we actualy prove is slightly more). The arguments involve an application of the theory of Bernstein-Zelevinsky derivatives.
\begin{prop}\label{prop::Speh--Nece+MaxiLevi}
	Let $\pi = \text{Sp}(\Delta,\shskip l)$ be an essentially Speh representation of $G_n$, where $\Delta$ is an essentially square-integrable representaion of $G_d$, $d > 1$, and $l$ is a positive integer. Assume that $\pi$ is $H_{p,q}$-distinguished or $(H_{p,q},\mu_{-1/2})$ for two positive integers $p$, $q$, $p + q = n$. Then the degree $d$ of $\Delta$ is even, and $\Delta$ is $H_{d/2,d/2}$-distinguished; also one has $p = q$. 
	%	\begin{itemize}
	%		\item[(1)]  the degree $d$ of $\Delta$ is even, and $\Delta$ is $H_{d/2,d/2}$-distinguished;
	%		\item[(2)]  $p = q$.
	%	\end{itemize}
\end{prop}
\begin{proof}
	We prove this by induction on $l$. The case $l = 1$ follows from Proposition \ref{prop::starting point}. Suppose that $\pi$ is $(H_{p,q},\mu_a)$-distinguished with $a = 0$ or $-1/2$. By Proposition \ref{prop::ladder-Speh--selfdual} we know that $\pi$ is self-dual, hence $\Delta$ is also self-dual. Note that $\pi$ is irreducible. By Lemma \ref{lem::symmetry}, we may assume that $p \geqslant q$. By the assumption on $\pi$, we have
	\begin{align*}
		\Hom_{P_n \cap H_{p,q}} (\pi|_{P_n}, \mu_a) \neq 0,
	\end{align*}
    where $a = 0$ or $-1/2$. By \cite[\S 3.5]{BZ-I}, the restriction $\pi|_{P_n}$ of $\pi$ to $P_n$ has a filtration which has composition factors $(\Phi^+)^{i-1}\Psi^+ (\pi^{(i)})$, $i = 1,\cdots, h$, where $\pi^{(h)}$ is the highest derivative of $\pi$. We first analyze linear functionals on these factor sapces using the theory of Bernstein-Zelevinsky derivatives. 
	
	(1)\ When $i = 2k$ is even. If $ q > k$ and $p > k-1$, by applying \eqref{formula::Derivative-MainFormula} repeatly, we have
	\begin{align}\label{formula::i even-1}
		\Hom_{P_n \cap H_{p,q}} ((\Phi^+)^{i-1}\Psi^+ (\pi^{(i)}), \mu_a)  &\cong \Hom_{P_{n-i+1} \cap H_{q-k,p-k+1}}(\Psi^+(\pi^{(i)}),\mu_{-a-1/2})  \\
		& \cong \Hom_{ H_{q-k,p-k}} (\nu^{1/2}\pi^{(i)},\mu_{-a-1/2}). \nonumber
	\end{align} 
	Otherwise, there exists $i_0 \geqslant 0$ such that
	\begin{align}\label{formula::i even-2}
		\Hom_{P_n \cap H_{p,q}} ((\Phi^+)^{i-1}\Psi^+ (\pi^{(i)}), \mu_a)  &\cong \Hom_{P_{n-i+i_0+1} }((\Phi^+)^{i_0}\Psi^+(\pi^{(i)}),\mu_{a'}),
	\end{align}
	where $a' = a$ or $-a - 1/2$ depending on $i_0$ odd or even.
	
	(2)\ When $i = 2k + 1$ is odd. If $ q > k$ and $p > k$, by applying \eqref{formula::Derivative-MainFormula} repeatly, we have
	\begin{align}\label{formula::i odd-1}
		\Hom_{P_n \cap H_{p,q}} ((\Phi^+)^{i-1}\Psi^+ (\pi^{(i)}), \mu_a)  &\cong \Hom_{P_{n-i+1} \cap H_{p-k,q-k}}(\Psi^+(\pi^{(i)}),\mu_{a})  \\
		& \cong \Hom_{H_{p-k,q-k-1}} (\nu^{1/2}\pi^{(i)},\mu_{a}).\nonumber
	\end{align} 
	Otherwise, there exists $i_0 \geqslant 0$ such that
	\begin{align}\label{formula::i odd-2}
		\Hom_{P_n \cap H_{p,q}} ((\Phi^+)^{i-1}\Psi^+ (\pi^{(i)}), \mu_a)  &\cong \Hom_{P_{n-i+i_0+1} }((\Phi^+)^{i_0}\Psi^+(\pi^{(i)}),\mu_{a'}),
	\end{align}
	where $a' = a$ or $-a - 1/2$ depending on $i_0$ even or odd.
	
	We claim that the factor spaces corresponding to non-highest derivatives contribute nothing, that is, we have
	\begin{align}\label{formula::Speh distinction-zero contribution}
		\Hom_{P_n \cap H_{p,q}} ( (\Phi^+)^{i - 1} \Psi^+ (\pi^{(i)}), \mu_a)  =  0,\quad \text{for all } 1 \leqslant i < h.
	\end{align}
	
	We shall discuss separately according to $i$ is even or odd, $a = 0$ or $-1/2$. Note first that, by Lemma \ref{lem::ladder-derivative-left aligned}, when $1 \leqslant i < h$, the $i$-th derivative $\pi^{(i)}$ is either $0$ or a ladder representation of the form
	\begin{align}\label{formula::ladder-derivative-Speh}
		\langlands(\Delta_1 \times \nu^{(l-3)/2}\Delta \times \cdots \times \nu^{(1-l)/2}\Delta),
	\end{align} 
	where $\Delta_1$ is a subsegment of $\nu^{(l-1)/2}\Delta$ obtained by discarding the first few terms. In particular, $\pi^{(i)}$ is either $0$ or an irreducible representation. Thus, if we are in the case where \eqref{formula::i even-2} or \eqref{formula::i odd-2} holds, then 
	\begin{align*}
		\Hom_{P_n \cap H_{p,q}} ((\Phi^+)^{i-1}\Psi^+ (\pi^{(i)}), \mu_a)  &\cong \Hom_{P_{n-i+i_0+1} }((\Phi^+)^{i_0}\Psi^+(\pi^{(i)}),\mu_{a'}) \\
		&= 0,
	\end{align*}
	as the representation $(\Phi^+)^{i_0}\Psi^+(\pi^{(i)})$ is either $0$ or an irreducible representation of $P_{n-i+i_0+1}$ that is not one dimensional by \cite[3.3 Remarks]{BZ-I}.
	
	Now we deal with the case where \eqref{formula::i even-1} or \eqref{formula::i odd-1} holds. Note that, from \eqref{formula::ladder-derivative-Speh}, $\nu^{1/2}\pi^{(i)}$ either is $0$ or can be realized as the unique irreducible quotient of a representation of the form $\nu^{1/2}\Delta_1 \times \speh(\Delta,l-1)$ with $\Delta_1$ as above. We discuss as follows.
	
	Case (1) where $a = 0$ and $i = 2k$ is even. By \eqref{formula::i even-1}, it suffices to show that 
	\begin{align}\label{formula::ladder--case 1}
		\Hom_{ H_{q-k,p-k}}( \nu^{1/2}\Delta_1 \times \speh(\Delta,l-1), \shskip \mu_{-1/2}) = 0.
	\end{align}
	Assume, on the contrary, that $\nu^{1/2}\Delta_1 \times \speh(\Delta,l-1)$ is $(H_{q-k,p-k},\mu_{-1/2})$-distinguished. As $\Delta$ is self-dual, $\mathbf{e}(\speh(\Delta,l-1))^{\vee} = \mathbf{b}(\speh(\Delta,l-1)) \neq \mathbf{b}(\nu^{1/2}\Delta_1)$. So, by Proposition \ref{prop::Geometric Lemma--Products Products},  
	\begin{align*}
		 \text{ $\nu^{1/2}\Delta_1$ is $(H_{r,s},\mu_{(s - p- r + q-1)/2})$-distinguished }
	\end{align*}
   and 
   \begin{align*}
   	\text{$\speh(\Delta,l-1)$ is $(H_{q - k -r, \shskip p - k -s},\mu_{(s - r -1)/2})$-distinguished}
   \end{align*}
   for some nonnegative integers $r$ and $s$. If the degree of $\nu^{1/2}\Delta_1$ is greater than $1$, then $\nu^{1/2}\Delta_1$ is self-dual by Proposition \ref{prop::starting point}. This is absurd because the central character of $\nu^{1/2}\Delta_1$ has positive real part; If the degree of $\nu^{1/2}\Delta_1$ is $1$, then $(r,s) = (1,0)$ or $(0,1)$. If $r =1$ and $s = 0$, then $\speh(\Delta, l -1)$ is $(H_{q-k-1,p-k},\mu_{-1})$-distinguished. Thus we have $ p = q - 1$ by Corollary\ref{cor::Speh--Nece+MaxiLevi--CharacterInGeneralPosition}. This is absurd as we have assumed that $p \geqslant q$; If $r = 0$ and $ s = 1$, then $\nu^{1/2}\Delta_1$ is the character $\nu^{(p-q)/2}$ of $G_1$ and $\speh(\Delta,l-1)$ is $(H_{q-k,\shskip p-k-1},\mathbf{1})$-distinguished. So, by induction hypothesis, we have $p - 1 = q$. This implies that $\mathbf{e}(\nu^{(l-1)/2}\Delta) = \mathbf{e}(\Delta_1) = \mathbf{1}$, the trivial character of $G_1$. This is impossible as $\Delta$ is self-dual and its degree $d$ is greater than $1$.
	
	Case (2) where $a = 0$ and $i = 2k+1$ is odd. In this case we see easily that
	\begin{align}\label{formula::ladder--case 2}
		\Hom_{H_{p-k,q-k-1}} (\nu^{1/2}\pi^{(i)},\mathbf{1}) = 0,
	\end{align}
	as the central character of $\nu^{1/2}\pi^{(i)}$ has positive real part when $i < h$.
	
	The arguments for the remaining two cases where $a = -1/2$, $i$ is even or odd are similar to those of the above two cases and are omitted here. So we have proved \eqref{formula::Speh distinction-zero contribution}.
	
	By Lemma \ref{lem::ladder-derivative-left aligned}, we know that the highest derivative of $\pi$ is $\pi^{(d)}$ and $\nu^{1/2}\pi^{(d)}  = \speh(\Delta,l-1)$. Now we have 
	\begin{align}\label{formula::ladder--maxi Levi-highest derivative part}
		\Hom_{P_n \cap H_{p,q}} ( (\Phi^+)^{d-1} \Psi^+ (\pi^{(d)}), \mu_a)  \neq  0,
	\end{align}
	where $a = 0$ or $-1/2$. We analyze the left hand side of \eqref{formula::ladder--maxi Levi-highest derivative part} as above. The cases \eqref{formula::i even-2} and \eqref{formula::i odd-2} cannot happen by the same arguments as above. The case \eqref{formula::i odd-1} cannot happen by induction hypothesis and the fact that $p \geqslant q$. So the only possible case is when \eqref{formula::i even-1} holds, that is, $d$ is even and 
	\begin{align*}
		\Hom_{P_n \cap H_{p,q}} ((\Phi^+)^{i-1}\Psi^+ (\pi^{(d)}), \mu_a)  \cong \Hom_{ H_{q-k,p-k}} (\speh(\Delta,l-1),\mu_{-a-1/2}). 
	\end{align*}
	Note that when $a = 0$ or $-1/2$, $-a-1/2 = -1/2$ or $0$. Thus we are done by induction hypothesis.
\end{proof}

 We have the following generalization of Corollary \ref{cor::discrete series--max Levi} to essentially Speh representations.

\begin{cor}\label{cor::Speh + distinguishe = max Levi G_m,m}
	Let $\pi$ be an essentially Speh representation of $G_n$ that is not one dimensional. If $\pi$ is $(H_{p,q},\mu_a)$-distinguished for two positive integers $p$, $q$ with $p + q = n$ and $a \in \BR$, then we have $p = q$.
\end{cor}
\begin{proof}
	The case $a \neq 0$ is Corollary \ref{cor::Speh--Nece+MaxiLevi--CharacterInGeneralPosition}. The case $a = 0$ follows from Proposition \ref{prop::Speh--Nece+MaxiLevi}.
\end{proof}
\begin{rem}
	We postpone the proof of the other direction of Theorem \ref{thm::1 Speh} in Section \ref{section::Speh case}.
\end{rem}

\subsection{Distinguished left aligned representations}

The results of this subsection are used only in Section \ref{section::general case} where we classify distinguished representations that are products of Speh representations. The analysis in this subsection is quite involved; the readers can skip it for the fisrt reading. 

The purpose of this subsection is to show the following
\begin{prop}\label{prop::left aligned + distinguish = Speh}
		Let $\pi$ be a left aligned (resp. right aligned) representation of $G_n$. If $\pi$ is $(H_{p,q},\mu_{(p-q)/2})$ (resp. $(H_{p,q},\mu_{(q-p)/2})$)-distinguished for two nonnegative integers $p$, $q$, $p + q =n$, then $\pi$ is an essentially Speh representation.
\end{prop}

%The reader is advised to skip this section for the first reading, go directly to the proof of the classification of distinguished unitary representations of Arthur type %in Theorem \ref{thm::proof+classification+UnitaryDual}, and then read this section when necessary.

We need the following technical lemmas. When the supercuspidal representations in the support of the left aligned representation have degree greater than $1$, we can prove slightly more.

\begin{lem}\label{lem::ladder distinction-same length}
	Let $\rho \in \mathscr C(G_d)$, $d > 1$, and $\mathfrak{m} = \{\Delta_1, \cdots, \Delta_t \} \in \mathscr O_{\rho}$ be a ladder. Assume that $\pi = \langlands(\mathfrak{m})$ is a decreasing or  an increasing ladder representation of $G_n$. If $\pi$ is $(H_{p,q},\mu_a)$-distinguished for two positive integers $p$, $q$, $p + q =n$ and some $a \in \BR$, then all the $l(\Delta_i)$'s are the same. Moreover, $\pi$ is self-dual.
\end{lem}
\begin{proof}
	Note that $\pi$ is irreducible. By Lemma \ref{lem::symmetry}, passing to contragradient if necessary, we may assume that $l(\Delta_1) \leqslant l(\Delta_2) \leqslant \cdots \leqslant l(\Delta_t)$.By our assumption, the representation $\Delta_1 \times \cdots \times \Delta_t$ is $(H_{p,q},\mu_a)$-distinguished. We now appeal to Proposition \ref{prop::Geometric Lemma-Main}. Write $\Delta_i = \Delta([a_i,b_i]_{\rho})$, $i = 1, 2, \cdots,t$. Note that by our assumption that $d > 1$, Case A1 and Case B1 cannot happen.
	
	Case A2. In this case We have $a_t  \leqslant c_t < b_t$, and $\Delta( [a_t,c_t]_{\rho})$ is self-dual. Thus we have $\nu^{a_t} \rho \cong \nu^{-c_t} \rho^{\vee}$, and consequently $(a_t + c_t)d + 2 \Re(w_{\rho}) = 0$. We also have $\Delta( [c_t+1,b_t]_{\rho} )^{\vee} \cong \Delta([a_i,c_i]_{\rho})$ for some $i < t$ and $c_i \geqslant a_i$. Thus we get $\nu^{a_i}\rho \cong \nu^{-b_t}\rho^{\vee}$, and then $(a_i + b_t)d + 2 \Re (w_{\rho}) = 0$. But this is absurd because $a_i > a_t$ and $b_t > c_t$.
	
	Case B2. In this case we have $c_t = b_t$,  and $\Delta( [a_t,b_t]_{\rho})$ is self-dual. Thus we have $\nu^{a_t} \rho \cong \nu^{-b_t} \rho^{\vee}$, and consequently $(a_t + b_t)d + 2 \Re (w_{\rho}) = 0$. We also have $\Delta_1 \times \cdots \times \Delta_{t-1}$ is $(H_{p',q'},\mu_{a'})$-distinguished for some $p'$, $q'$ and $a'$. If $t = 1$, there is nothing to be proved. If $t > 1$, by Corollay \ref{cor::Geometric Lemma-Main}, we get that $(\nu^{b_{t-1}}\rho)^{\vee} \cong \nu^{a_i}\rho$ for some $1 \leqslant i \leqslant t-1$. Thus we get $(a_i + b_{t-1})d + 2 \Re (w_{\rho}) = 0$. This is absurd because $a_i > a_t $ and $b_{t-1} > b_t$. 
	
	So the only possible case is Case C. We then have $\Delta([a_t,b_t])^{\vee} \cong \Delta([a_i,c_i] )$ for $i < t$ and certain $ a_i \leqslant c_i \leqslant b_i$. Note that, by our assumption, we have $l(\Delta_i) \leqslant l(\Delta_t)$. Thus we have $l(\Delta_i) = l(\Delta_t)$. We claim that $i = 1$. If so,  all $l(\Delta_i)$'s will be the same by our assumption. Indeed, if $i > 1$, consider the $(H_{p,q},\mu_a)$-distinguished representaion
	\begin{align*}
		(\Delta_1 \times \cdots  \times \Delta_{i-1}) \times (\Delta_i \times \cdots \times \Delta_t).
	\end{align*}
	By Propostion \ref{prop::Geometric Lemma--Products Products}, either we have $\mathbf{e}(\Delta_{i-1})^{\vee} \cong \mathbf{b}(\Delta_a)$ with $1 \leqslant a \leqslant t-1$, or we have $\mathbf{e}(\Delta_c)^{\vee} \cong \mathbf{b}(\Delta_b)$ with $1 \leqslant b \leqslant t-1$ and $i \leqslant c \leqslant t$. We then get a contradiction as in Case A2 or B2. The assertion on the self-dualness property follows from a repeated analysis as above.
\end{proof}

If we drop the assumption that $d > 1$, the argument becomes complicated by the possible occurrence of Case A1 or Case B1 when applying Proposition \ref{prop::Geometric Lemma-Main}. We have the following result
on the shape of right aligned representations when it is distinguished. 

\begin{lem}\label{lem::distinguished right aligned--shape}
	Let $\rho$ be a character of $G_1$, and $\mathfrak{m} \in \mathscr O_{\rho}$ be a ladder. Assume that $\pi = \langlands(\mathfrak{m})$ is a right aligned representation of $G_n$. If $\pi$ is $(H_{p,q},\mu_a)$-distinguished for two nonnegative integers $p$, $q$, $p + q =n$ and some $a \in \BR$, then either
	
	(1)$\ $we have
	\begin{align}\label{formula::distinguished right aligned--1}
		\frm  =  \{ \Delta_1,\cdots,\Delta_{i_1}, \Delta_{i_1 + 1}, \cdots, \Delta_{i_1 + i_2}, \Delta_{i_1 + i_2 +1}, \cdots,\Delta_{i_1 + i_2 +i_3} \}
	\end{align}
	with $i_1$, $i_2$ and $i_3 \geqslant 0$, such that $l(\Delta_k) = 1$ when $1\leqslant k \leqslant i_1$, $l(\Delta_{i_1 + k})  = l > 1$ when $1 \leqslant k \leqslant i_2$, $l(\Delta_{i_1 + i_2 + k})) = l + 1$ when $1 \leqslant k\leqslant i_2$, and that $\mathbf{e}(\Delta_{i_1 + i_2 + i_3})^{\vee} \cong \mathbf{b}(\Delta_{i_1 + 1})$ (See Figure \ref{fig::1} for an example),
	
	or
	
	(2)$\ $we have 
	\begin{align}\label{formula::distinguished right aligned--2}
		\frm  =  \{\Delta_1,\cdots, \Delta_{i_1}, \Delta_{i_1 + 1}, \cdots ,\Delta_{i_1  + i_2}  \}
	\end{align} 
	with $i_1$ and $i_2 >0$, such that $l(\Delta_k) = 1$ when $1 \leqslant k \leqslant i_1$, $l(\Delta_{i_1 + k}) = 2$ when $1 \leqslant k \leqslant i_2$, and that $\mathbf{e}(\Delta_{i_1 + i_2})^{\vee} \cong \mathbf{b}(\Delta_1)$ (See Figure \ref{fig::2} for an example).
\end{lem}
\begin{proof}
	%	We will prove this by Proposition \ref{prop::Geometric Lemma-Main}. As the arguments are similar to that of Proposition \ref{prop::ladder distinction-same length}, we %give a brief treatment here. Let $\Delta_t$ be the bottom segement of $\frm$. When $l(\Delta_t) = 1$, then $\frm$ is of the form \eqref{formula::distinguished Right %Speh--1} with $i_2 = i_3 = 0$. Now we assume that $l(\Delta_t) > 1$. We have five cases according to Proposition \ref{prop::Geometric Lemma-Main}. Note that Case A2 and %B1 cannot happen.
	%	
	%	If in Case C, then $\frm$ is of the form \eqref{formula::distinguished right aligned--1} with $i_3 = 0$ and $i_2 > 0$; If in Case B2, then $\frm$ is of the form %\eqref{formula::distinguished right aligned--1} with $i_2 = 1$ and $i_3 = 0$.
	%	
	%	If in Case A1, we have two subcases. If $l(\Delta_t) > 2$, then $\frm$ is of the form \eqref{formula::distinguished right aligned--1} with $i_3 = 0$ and $i_2 > 0$, or %with $i_2 > 0$ and $i_3 > 0$; If $l(\Delta_t) = 2$, we see that $\pi$ is the unique irreducible quotient of $\pi_1 \times \pi_2$, where $\pi_1$ is a one dimensional %representation and $\pi_2$ is a Speh representation of length $2$. By Proposition \ref{prop::ladder-products of left aligned}, either $\frm$ is of the form %\eqref{formula::distinguished right aligned--2} with $i_1 >0$ and $i_2 > 0$, or of the form \eqref{formula::distinguished right aligned--1} with $i_3 = 0$.
	Write $\frm = \{\Delta_1,\cdots,\Delta_t \}$. If $l(\Delta_t) =1$, then $\pi$ is a one dimensional representation and $\frm$ is of the form \eqref{formula::distinguished right aligned--1} with $i_2 = i_3 = 0$. If $l(\Delta_t) = l(\Delta_1) =2 $, then $\pi$ is an essentially Speh representation. It follows from Proposition \ref{prop::ladder-Speh--selfdual} that $\frm$ is of the form \eqref{formula::distinguished right aligned--1} with $i_1 = i_3 = 0$. If $l(\Delta_t) = 2$ and $l(\Delta_1) = 1$, then $\pi$ can be realized as the unique irreducible quotient of $\pi_1 \times \pi_2$, where $\pi_1$ is a one dimensional representation and $\pi_2$ is an essentially Speh representation of length $2$. Thus $\pi_1 \times \pi_2$ is $(H_{p,q},\mu_a)$-distinguished. By Proposition \ref{prop::ladder-products of left aligned}, $\frm$ is either of the form \eqref{formula::distinguished right aligned--2} (Case A), or of the form \eqref{formula::distinguished right aligned--1} with $i_3 = 0$, $i_1 > 0$, $i_2 > 0$ and $l = 2$ (Case B and Proposition \ref{prop::ladder-Speh--selfdual}). Note that here Case C is impossible by our assumption on $\pi_1$ and $\pi_2$. If $l(\Delta_t) > 2$, then we apply Proposition \ref{prop::Geometric Lemma-Main} to the product $\Delta_1 \times \cdots \times \Delta_t$ and discuss case by case. Note first that Case A2 cannot happen by similar arguments as those in Lemma \ref{lem::ladder distinction-same length}; Case B1 cannot happen by our assumption on $\Delta_t$. In the remaining cases, it follows from Corollary \ref{cor::Geometric Lemma-Main}, Proposition \ref{prop::Geometric Lemma--Products Products} and arguments similar to those in Lemma \ref{lem::ladder distinction-same length} that $\frm$ is of the form \eqref{formula::distinguished right aligned--1}. 
\end{proof}

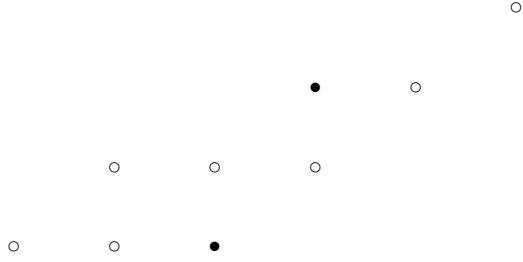
\begin{figure}[h]
	\begin{center}
		\begin{tikzcd}
			&       &         &         &       & \circ \\
			&       &         & \bullet & \circ &       \\
			& \circ & \circ   & \circ   &       &       \\
			\circ & \circ & \bullet &         &       &      
		\end{tikzcd}
		
		\caption{An example of a ladder of the form \eqref{formula::distinguished right aligned--1} with $i_1 = i_2 = 1$ and $i_3 = 2$}
		\label{fig::1}
	\end{center}
	
\end{figure}

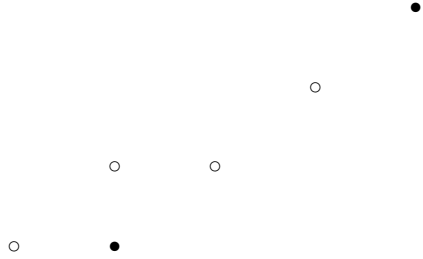
\begin{figure}[h!]
	%\begin{tikzcd}
	%&         &       &       &       & \bullet \\
	%&         &       &       & \circ &         \\
	%&         & \circ & \circ &       &         \\
	%& \circ   & \circ &       &       &         \\
	%\circ & \bullet &       &       &       &        
	%\end{tikzcd}
	\begin{center}
		\begin{tikzcd}
			&         &       &       & \bullet \\
			&         &       & \circ &         \\
			& \circ   & \circ &       &         \\
			\circ & \bullet &       &       &        
		\end{tikzcd}
		\caption{An example of a ladder of the form \eqref{formula::distinguished right aligned--2} with $i_1 = 2$ and $i_2 = 2$}
		\label{fig::2}
	\end{center}
\end{figure}

The following lemma is a simple consequence of Lemma \ref{lem::one dimensional--distinction}.  
\begin{lem}\label{lem::one dimensional--special condition}
	Let $\pi$ be a one dimensional representation of $G_n$. If $\pi$ is $(H_{p,q},\mu_{(q-p)/2})$-distinguished with $p + q = n$, then $\pi$ is either the trivial character $\mathbf{1}$ of $G_n$ or the character $\nu^{-n/2}$ of $G_n$. In particular, $\mathbf{b}(\pi)$ is either $\nu^{(n-1)/2}$ or $\nu^{-1/2}$ of $G_1$.
\end{lem}

As shown in Lemma \ref{lem::distinguished right aligned--shape}, there are two possibilities for the shape of distinguished right aligned representations. Now we remove one possibility if we impose some restriction on the subscripts $(p,q,a)$. 

\begin{lem}\label{lem::distinguished right aligned--False--special condition}
	Keep the notation as in Lemma \ref{lem::distinguished right aligned--shape}, let $\pi = \langlands(\frm)$ with $\frm$ of the form \eqref{formula::distinguished right aligned--2}. Then $\pi$ cannot be $(H_{p,q},\mu_{(q-p)/2})$-distinguished.
\end{lem}
\begin{proof}
	We assume on the contrary that $\pi$ is $(H_{p,q},\mu_{(q-p)/2})$-distinguished. By part $(1)$ of Lemma \ref{lem::symmetry}, we may assume that $p \leqslant q$. Note that $\pi$ can be realized as the unique quotient of $\pi_1 \times \pi_2$, where $\pi_1$ is a one dimensional representation, $\pi_2$ is an essentially Speh representation of length $2$, and $\mathbf{e}(\pi_2)^{\vee} \cong \mathbf{b}(\pi_1)$. Thus, $\pi_1 \times \pi_2$ is $(H_{p,q},\mu_{(q-p)/2})$-disintuished. By Proposition \ref{prop::ladder-products of left aligned}, there exist divisions of $\pi_1$ and $\pi_2$, $\pi_1 = \pi_1' \sqcup \pi_1''$ and $\pi_2 = \pi_2' \sqcup \pi_2''$ respectively, such that, among other things, $\pi_1''$ is $(H_{r,s},\mu_{(s-r)/2})$-distinguished for two nonnegative integers $r$ and $s$. Note that $\pi_1'$ is not the trivial representation of $G_0$ by our assumption on $\pi$ and Proposition \ref{prop::ladder-Speh--selfdual}. We shall discuss further according to the values of $r$ and $s$.
	
	(1)$\ $If exactly one of $r$ and $s$ is $0$, then $\pi_1''$ is the character $\nu^{-n_1''/2}$ of $G_{n_1''}$. Thus $\mathbf{b}(\pi_1'') = \nu^{-1/2}$, $\mathbf{e}(\pi_1') = \nu^{1/2}$. By  Proposition \ref{prop::ladder-products of left aligned}, we also have $\pi_1'^{\vee} \cong \pi_2''$. Thus $\mathbf{b}(\pi_2'') = \mathbf{e}(\pi_1')^{\vee} = \nu^{-1/2} = \mathbf{b}(\pi_1'')$, which is absurd.
	
	(2)$\ $If $r > 0$ and $s > 0$, then $\pi_1''$ is the character $\mathbf{1}$ of $G_{2r}$, that is, $\mathbf{b}(\pi_1'') = \nu^{r - 1/2}$ and $\mathbf{e}(\pi_1'') = \nu^{-r + 1/2}$. So, $\mathbf{e}(\pi_1') = \mathbf{b}(\pi_1'')  + 1 = \nu^{r+1/2}$. By Proposition \ref{prop::ladder-products of left aligned}, we have $\pi_1'^{\vee} \cong \pi_2''$. So $\mathbf{b}(\pi_2'') = \nu^{-r-1/2}$. By our assumption on the shape of $\mathfrak{m}$, this implies that $\pi_2'$ is also a one dimensional representation which, by Proposition \ref{prop::ladder-products of left aligned}, is $(H_{r',s'},\mu_{(r' -s')/2+q - p})$-distinguished for certain nonnegative integers $r'$ and $s'$ and that $\mathbf{b}(\pi_2') = \mathbf{b}(\pi_2'') - 1 = \nu^{-r-3/2}$. One of $r'$ and $s'$ has to be $0$. Recall that we have assumed that $p \leqslant q$. We then see easily that $\pi_2'$ is the character $\nu^{p-q + n' /2}$ of $\GL_{n'}$. Note that we have an equality of central characters, $\omega_{\pi_2'} = \omega_{\pi}$. This implies that
	\begin{align*}
		n'(p-q + n' /2) = - (p-q)^2  / 2.
	\end{align*}
    So, $n' = q - p$ and $\mathbf{b}(\pi_2') = \nu^{-1/2}$. This is absurd as we have shown that $\mathbf{b}(\pi_2') = \nu^{-r - 3/2}$ with $r$ a positive integer.
	
	(3)$\ $If $r = s =0$, we have two subcases according to whether or not $\pi_2'$ is a one dimensional representaiton. If it is, we get a contradiction by exactly the same arguments as in $(2)$ with $r$ being repalced by $0$. If it is not, it follows from the duality relations in Lemma \ref{lem::distinguished right aligned--shape}, applied to the contragradient of $\pi_2'$, that $\mathbf{b}(\pi_2')^{\vee}$ is either $\mathbf{b}(\pi_2'') + 1$ or $\mathbf{e}(\pi_2'') - 1$. But, note that $\mathbf{b}(\pi_2') = \mathbf{e}(\pi_1') - 2$. It follows from the relation $\pi_1'^{\vee} \cong \pi_2''$ that $\mathbf{b}(\pi_2') ^{\vee} = \mathbf{b}(\pi_2'') + 2$. This is absurd as $\pi_2''$ is one dimensional and $\mathbf{b}(\pi_2'') \geqslant \mathbf{e}(\pi_2'')$.
\end{proof}

\begin{proof}[Proof of Proposition \ref{prop::left aligned + distinguish = Speh}]
	By part $(2)$ of Lemma \ref{lem::symmetry}, we only need to prove the statement for left aligned representations.  We may further assume that $p \leqslant q$ by part $(1)$ of Lemma \ref{lem::symmetry}. By Lemma \ref{lem::ladder distinction-same length}, we may write $\pi = \langlands(\mathfrak{m})$ with $\mathfrak{m} \in \mathscr O_{\rho}$ a ladder and $\rho$ a character of $G_1$.  By considering the contragredient $\pi^{\vee} = \langlands(\frm^{\vee})$, we see from Lemma \ref{lem::distinguished right aligned--False--special condition} that  $\frm^{\vee}$ is of the form \eqref{formula::distinguished right aligned--1}. So, we may write
	\begin{align*}
		\frm  =  \{ \Delta_1,\cdots,\Delta_{i_1}, \Delta_{i_1 + 1}, \cdots, \Delta_{i_1 + i_2}, \Delta_{i_1 + i_2 +1}, \cdots,\Delta_{i_1 + i_2 +i_3} \}
	\end{align*}
	with $i_1$, $i_2$ and $i_3 \geqslant 0$, such that $l(\Delta_k) = l_1 > 2$ when $1\leqslant k \leqslant i_1$, $l(\Delta_{i_1 + k})  = l_1 -1$ when $1 \leqslant k \leqslant i_2$, $l(\Delta_{i_1 + i_2 + k})) = 1 $ when $1 \leqslant k\leqslant i_3$, and that $\mathbf{e}(\Delta_{i_1 + i_2 })^{\vee} \cong \mathbf{b}(\Delta_1)$.
	
	We may as well assume that $i_1$ and $i_2$ are not all zero. Our first step is to show that $i_3 = 0$. If not so, we realize $\pi$, in the obvious way, as the unique irreducible quotient of $\pi_1 \times \pi_2 \times \pi_3$ with $\pi_i$ an essentially Speh representation for each $i$, such that $\pi_3$ is a character of $G_{n_3}$, $n_3 > 0$, and that at least one of $\pi_1$ and $\pi_2$ is not the trivial representation of $G_0$. By our assumption on $\pi$, the representation $\pi_1 \times \pi_2 \times \pi_3$ is $(H_{p,q},\mu_{(p-q)/2})$-distinguished. Note that as $i_3 > 0$, $\mathbf{e}(\pi_3)$ is not dual to $\mathbf{b}(\pi_1)$ or $\mathbf{b}(\pi_2)$. So, by Proposition \ref{prop::Geometric Lemma--Products Products}, $\pi_3$ is $(H_{r,s},\mu_{(r-s)/2})$-distinguished with respect to two nonnegative integers $r$ and $s$. As $\pi_3$ is one dimensional, $\pi_3$ is either the trivial represntation $\mathbf{1}$ of $G_{n_3}$ or the character $\nu^{n_3/2}$ of $G_{n_3}$. In particular, $\mathbf{b}(\Delta_{i_1 + i_2 + 1}) = \nu^{(n_3 - 1)/2}$ or $\nu^{n_3  -  1/2}$. But this will contradict with the fact that $\mathbf{e}(\Delta_{i_1 + i_2 })^{\vee} \cong \mathbf{b}(\Delta_1)$.
	
	Our next step is to show that $i_1 = 0$ or $i_2 = 0$. Assume on the contrary that $i_1 > 0$ and $i_2 > 0$. By our assumption on $\pi$, the representation $\pi^{\vee} = \langlands(\frm^{\vee})$ is $(H_{p,q},\mu_{(q-p)/2})$-distinguished. Thus, the representation
	\begin{align*}
		\Delta_{i_1 + i_2}^{\vee}  \times \cdots \times \Delta_{i_1 + 1}^{\vee} \times \Delta_{i_1}^{\vee} \times \cdots \times \Delta_1^{\vee}
	\end{align*}
	is $(H_{p,q},\mu_{(q-p)/2})$-distinguished. By Proposition \ref{prop::Geometric Lemma-Main}, we deduce that $\mathbf{b}(\Delta_1^{\vee}) \cong \mathbf{e}(\Delta_1)^{\vee}$ is the character $\nu^{q-p + 1/2}$ or $\nu^{p-q + 1/2}$ of $G_1$. (This is the consequence of Case A1; Case A2 and Case B2 are eliminated by arguments similar to those in Lemma \ref{lem::ladder distinction-same length}; Case B1 and Case C are eliminated by our assumptions.) It follows easily from the condition $\mathbf{b}(\Delta_1) \cong \mathbf{e}(\Delta_{i_1  + i_2})^{\vee}$ and the assumption $p \leqslant q$ that $\mathbf{e}(\Delta_1)^{\vee} = \nu^{p -q + 1/2}$. Hence we have $\mathbf{e}(\Delta_1)  =  \nu^{q - p - 1/2}$.
	
	We show that $i_1 = q - p$ and $\mathbf{e}(\Delta_{i_1})  =  \nu^{1/2}$ by consideration on the central character of $\pi$. In fact, on the one hand, we see from the assumption on $\frm$ and the fact  $\mathbf{e}(\Delta_1) = \nu^{ q - p -1/2}$ that the central character $w_{\pi}$ of $\pi$ is $\nu^a$ where $a =(q-p)i_1 - i_1^2/2$; on the othe hand, as $\pi$ is $(H_{p,q},\mu_{(p-q)/2})$-distinguished, we have $w_{\pi} = \nu^{a'}$ where $a' = (q-p)^2 /2 $. Thus the assertion follows. Also, from the fact that $\mathbf{e}(\Delta_{i_1 + i_2 })^{\vee} \cong \mathbf{b}(\Delta_1)$, we get that $\mathbf{b}(\Delta_1) = \nu^{i_2 + 1/2}$. Thus, $l(\Delta_1) = q - p -i_2 >2$, in particular $i_1 > i_2$.
	
	Now, as in the first step, we have that $\pi_1 \times \pi_2$ is $(H_{p,q},\mu_{(p-q)/2})$-distinguished, where $\pi_1 = \langlands(\Delta_1,\cdots,\Delta_{i_1})$ and $\pi_2 = \langlands(\Delta_{i_1 + 1}, \cdots, \Delta_{i_1  + i_2})$. We appeal to Proposition \ref{prop::ladder-products of left aligned}, and claim that Case A and Case B cannot happen. In fact, if Case A or Case B happens, there will be a division of $\pi_2$ as $\pi_2'$ and $\pi_2''$, where $\pi_2'$ is not the trivial representation of $G_0$, such that $\pi_2'$ is $(H_{r,s},\mu_{(r-s)/2})$-distinguished for two nonnegative integers $r$ and $s$. In particular, the central character $w_{\pi_2'}$ of $\pi_2'$ has nonnegative real part. But this will contradict with the fact that $\mathbf{e}(\Delta_{i_1 + 1}) = \nu^{-3/2}$. So, there exists a division of $\pi_1$ as two ladder representations $\pi_1'$ and $\pi_1''$ such that $\pi_2 \cong \pi_1'^{\vee}$ and that $\pi_1''$ is $(H_{p-n_2,q-n_2},\mu_{(p-q)/2})$-distinguished. Note that $\pi_1''$ is a right aligned representation, and is not a one dimensional representation due to the fact that $i_1 > i_2$. By Lemma \ref{lem::distinguished right aligned--shape}, we then get a contradiction as we can check easily that the ladder $\frm_1''$ of $\pi_1''$ is not of the form \eqref{formula::distinguished right aligned--1} or \eqref{formula::distinguished right aligned--2}.
\end{proof}

%\begin{prop}\label{prop::left aligned+Distinguish= Speh G1}
%	Let $\rho$ be a character of $G_1$, and $\mathfrak{m} \in \mathscr O_{\rho}$ be a ladder. Assume that $\pi = \langlands(\mathfrak{m})$ is a left aligned %representation of $G_n$. If $\pi$ is $(H_{p,q},\mu_{(p-q)/2})$-distinguished for two nonnegative integers $p$, $q$, $p + q =n$, then $\pi$ is an essentially %Speh representation. 
%\end{prop}

%Finally, combining Proposition \ref{prop::ladder distinction-same length} and \ref{prop::left aligned+Distinguish= Speh G1}, we get
%\begin{thm}\label{thm::left aligned + distinguish = Speh}
%	Let $\pi$ be a left aligned (resp. right aligned) representation of $G_n$. If $\pi$ is $(H_{p,q},\mu_{(p-q)/2})$ (resp. %$(H_{p,q},\mu_{(q-p)/2})$)-distinguished for two nonnegative integers $p$, $q$, $p + q =n$, then $\pi$ is an essentially Speh representation.
%\end{thm}

\section{Distinction in the Unitary Dual}\label{section::Classification-Unitary Dual}

\subsection{The case of Speh representations}\label{section::Speh case}

We now classify distinguished Speh representations in terms of distinguished discrete series. In fact, we will do it for essentially Speh representations. 

\begin{thm}\label{thm::proof+classification+Speh}
	Let $n = 2m$, and $\speh(\Delta,k)$ be an essentially Speh representation of $G_n$, where $\Delta$ is an essentially square-integrable representation of $G_d$ with $d > 1$, and $k$ is a positive integer. Then $\speh(\Delta,k)$ is $H_{m,m}$-distinguished if and only if $d$ is even and $\Delta$ is $H_{d/2,d/2}$-distinguished.
\end{thm}
\begin{proof}
	One direction has been proved in Proposition \ref{prop::Speh--Nece+MaxiLevi}. We now assume that $d$ is even and that $\Delta$ is $H_{d/2,d/2}$-distinguished. By \cite[Proposition 7.2]{Offen-ParabolicInduction-JNT}, which is based on the work of Blanc and Delorme \cite{Blanc-Delorme}, the representation
	\begin{align}\label{formula::product representation}
		\nu^{(k-1)/2}\Delta \times \nu^{(k-3)/2} \Delta \times \cdots \times \nu^{(1-k)/2}\Delta
	\end{align}
	is $H_{m,m}$-distinguished. (The distinguishedness of $\Delta$ is unnecessary when $k$ is even).We have the following exact sequence of representations of $G_n$,
	\begin{align}
		0 \ra \mathcal{K} \ra \nu^{(k-1)/2}\Delta \times \nu^{(k-3)/2} \Delta \times \cdots \times \nu^{(1-k)/2}\Delta  \ra \speh(\Delta,k) \ra 0,
	\end{align}
	where the kernel $\mathcal{K} = \sum_{i=1}^{k-1} \mathcal{K}_i$ is given explicitly in Proposition \ref{prop::ladder--kernel description}. To show that $\speh(\Delta,k)$ is $H_{m,m}$-distinguished, it suffices to show that each $\mathcal{K}_i$ is not $H_{m,m}$-distinguished. Write the representation \eqref{formula::product representation} as $\Delta([a_1,b_1]_{\rho}) \times \cdots \times \Delta([a_k,b_k]_{\rho})$, here the cuspidal representation $\rho$ is taken to be self-dual and thus $a_i$ and $b_i$, $i = 1, 2,\cdots,k$  need not be integers. So we have
	\begin{align*}
		& a_{i+1} = a_i - 1, \quad b_{i+1} = b_i - 1,\  i = 1, \cdots, k-1\\
		& a_i + b_{k+1-i} = 0,\  i = 1, \cdots, k.
	\end{align*}
	We further omit the subscript $\rho$ in the sequel. Recall that, by Proposition \ref{prop::ladder--kernel description},
	\begin{align*}
		\mathcal{K}_i  =  \Delta([a_1,b_1])  \times \cdots \times \Delta([a_{i+1},b_i]) \times \Delta([a_i,b_{i+1}]) \times \cdots \times \Delta([a_k,b_k]).   
	\end{align*}
	
	If $i+1 \leqslant (k+1)/2$ and $\mathcal{K}_i$ is $H_{m,m}$-distinguished, by applying Proposition \ref{prop::Geometric Lemma-Main} repeatly, we get that $\Delta([a_{i+1},b_i]) \times \Delta([a_i,b_{i+1}]) \times \cdots \times \Delta([a_{k+1-i},b_{k+1-i}])$ is $H_{m',m'}$-distinguished for certain $m'$. (In each step, only Case C is possible.) When we apply Proposition \ref{prop::Geometric Lemma-Main} once again, still, only Case C is possible. But this is absurd as $l(\Delta([a_i,b_{i+1}])) < l(\Delta)$. Similar arguments can show that $\mathcal{K}_i$ is not $H_{m,m}$-distinguished if $i \geqslant (k+1)/2$. 
	
	The remaining case is when $k$ is even and $i = k/2$. In what follows, to save notation, we sometimes write $H$-distinguished for $H_{m',m'}$-distinguished when there is no need to address $m'$. If $\mathcal{K}_i$ is $H_{m,m}$-distinguished, by applying Proposition \ref{prop::Geometric Lemma-Main} repeatly, we get that $\Delta([a_{i+1},b_i]) \times \Delta([a_i,b_{i+1}])$ is $H$-distinguished.  This in turn implies that both $\Delta([a_{i+1},b_i])$ and $\Delta([a_i,b_{i+1}])$ are $H$-distinguished by Proposition \ref{prop::Geometric Lemma--Products Products}. Let us write $\Delta = \steinberg(\rho,l)$. Then by our assumption on $i$, we have $\Delta([a_i,b_{i+1}]) = \steinberg(\rho,l-1)$ and $\Delta([a_{i+1},b_i]) = \steinberg(\rho,l+1)$. By \cite[Theorem 6.1]{Matringe-Linear+Shalika-JNT}, we can conclude that $\steinberg(\rho,l)$ is $H$-distinguished if and only if $\steinberg(\rho,l-1)$ (or $\steinberg(\rho,l+1)$) is not $H$-distinguished. Actually, as $\rho$ is self-dual, the $L$-function $L(s,\phi(\rho)\otimes \phi(\rho))$ has a simple pole at $s = 0$, where $\phi(\rho)$ is the Langlands parameter of $\rho$. By the factorization
	\begin{align*}
		L(s,\phi(\rho)\otimes \phi(\rho))  =  L(s,\Lambda^2 \circ \phi(\rho)) \cdot L(s,\textup{Sym}^2 \circ \phi(\rho)),
	\end{align*}
	we know that exactly one of the symmetric or exterior square $L$-factors of $\rho$ has a pole at $s=0$. The above conclusion then follows from \cite[Theorem 6.1]{Matringe-Linear+Shalika-JNT} where distinction of $\steinberg(\rho,l)$ is related to the pole of symmetric or exterior square $L$-facotrs of $\rho$ according to $l$ is even or odd. Thus by our assumption that $\Delta$ is $H_{d/2,d/2}$-distingusihed, we get that $\mathcal{K}_i$ is not $H_{m,m}$-distinguished. So we are done.
\end{proof}

\subsection{The general case}\label{section::general case}

We start with an auxiliary result, which is needed in one step of the proof of Theorem \ref{thm::proof+classification+UnitaryDual}. 

\begin{lem}\label{lem::cancellation result}
	Let $\pi = \pi_1 \times \cdots \times \pi_t$ be an irreducible unitary representation of $G_{2m}$ with each $\pi_i$ a Speh representation. Let $h$ be a positive integer. Assume that, for all of those $\pi_i$ such that $\supp(\pi_i)$ is contained in the cuspidal line $\BZ\nu^{-1/2}$, we have $\mathbf{b}(\pi_i) \leqslant \nu^{h-1/2}$. If the representation $\pi \times \nu^{-h/2}$ is $(H_{m,m+h},\mu_{h/2})$-distinguished, where $\nu^{-h/2}$ is viewed as a representation of $G_h$, then $\pi$ is $H_{m,m}$-distinguished.
\end{lem}
\begin{proof}
	A crucial fact, on which we rely, is that $\pi$ is a commutative product of Speh representations.
	Our first step is to show that we can reduce the proposition to the case that for all $i$,
	\begin{align}\label{formula::cancellation}
		\text{  the support $\supp(\pi_i)$ is contained in $\BZ \nu^{-1/2}$ and $\mathbf{b}(\pi_i) = \nu^{h-1/2}$.  }
	\end{align} 
	Indeed, write $\Pi = \pi_1 \times \cdots \times \pi_r$ and $\Pi'  =  \pi_{r+1} \times \cdots \times \pi_t \times \nu^{-h/2}$ where, $\pi_j$'s, $r+1 \leqslant j \leqslant t$, are all the representations in the Tadi\'{c} decomposition of $\pi$ that satisfy \eqref{formula::cancellation}. So $\Pi \times \Pi'$ is $(H_{m,m+h},\mu_{h/2})$-distinguished. By Proposition \ref{prop::Geometric Lemma--Products Products}, $\Pi$ is $(H_{m_1,m_1 + h_1},\mu_{h_1/2})$-distinguished and $\Pi'$ is $(H_{m-m_1,m-m_1+h-h_1},\mu_{(h+h_1)/2})$-distinguished for certain integers $m_1$ and $h_1$. Note that the central character of $\Pi$ has real part $0$. We have $h_1 = 0$. So, $\Pi$ is $H_{m_1,m_1}$-distinguished and $\Pi'$ is $(H_{m-m_1,m-m_1+h},\mu_{h/2})$-distinguished. The reduction then follows from Lemma \ref{lem::preservation-parabolic induction}.
	
	 We thus assume that $\pi = \pi_1 \times \cdots \times \pi_t$ with $\mathbf{b}(\pi_i) = \nu^{h-1/2}$ for all $i$. Moreover, we arrange the ordering of $\pi_i$'s such that $\mathbf{ht}(\pi_1) \geqslant \cdots \geqslant \mathbf{ht}(\pi_t)$. We prove the lemma by induction on $t$.
	
	As the representation $\pi \times \nu^{-h/2}$ is $(H_{m,m+h},\mu_{h/2})$-distinguished, by Proposition \ref{prop::ladder-products of left aligned}, there exist two representations $\sigma'$ and $\sigma''$ of dimension one, $\nu^{-h/2} = \sigma' \sqcup \sigma''$, such that, among other things, $\sigma'$ is $(H_{a,b},\mu_{h+(a-b)/2})$-distinguished for two nonnegative integers $a$ and $b$.
	
	(1). If $\sigma'$ is not the trivial representation of $G_0$, that is, $a$ and $b$ are not all zero, we have three cases. If $a > 0$ and $b > 0$, then by Lemma \ref{lem::one dimensional--distinction}, $\sigma'$ must be the trivial representation $\mathbf{1}$ of $G_{a+b}$. This is absurd as we have $\textbf{b}(\sigma') = \nu^{ - 1/2}$; If $a > 0$ and $b = 0$, then $\sigma'$ is the character $\nu^{h+a/2}$ of $G_a$. Thus $\textbf{b}(\sigma') = \nu^{h+a -1/2}$ which is absurd; If $a =0$ and $b>0$, we see easily that $a = 0$ and $b= h$, that is, $\sigma'$ is the character $\nu^{-h/2}$ of $G_h$. So, it follows from Case B of Proposition \ref{prop::ladder-products of left aligned} that $\pi$ is $H_{m,m}$-distinguished.
	
	(2). If $\sigma'$ is the trivial representaion of $G_0$, then we are in Case C of Proposition \ref{prop::ladder-products of left aligned}.Hence there exists $i$, $1 \leqslant i \leqslant t$ and a division of $\pi_i$ as two ladder representations $\pi_i'$ and $\pi_i''$, $\pi_i = \pi_i' \sqcup \pi_i''$, such that $\pi_i'$ is the character $\nu^{h/2}$ of $G_h$ and the representation
	\begin{align}\label{formula::auxiliary in general case--product}
		\pi_1 \times \cdots \times \pi_{i-1} \times \pi_i'' \times \pi_{i+1} \times \cdots \times \pi_t
	\end{align}
	is $(H_{m-h,m},\mu_{h/2})$-distinguished. We have two subcases. If $\pi_i$ is one dimensional, then $\pi_i$ must be the trivial representaion $\mathbf{1}$ of $G_{2h}$ as $\textbf{b}(\pi_i) = \nu^{h-1/2}$. Thus $\pi_i''$ is the character $\nu^{-h/2}$ of $G_h$. By Lemma \ref{lem::ladder--Commutativity},  $\nu^{-h/2} \times \pi_j = \pi_j \times \nu^{-h/2}$ for $j = 1,\cdots,t$. So we move $\pi_i''$ to the end of the product \eqref{formula::auxiliary in general case--product} and get by induction hypothesis that 
	\begin{align*}
	    \pi_1 \times \cdots \times \pi_{i-1} \times \pi_{i+1} \times \cdots \times \pi_k
	\end{align*}
	is $H_{m-h,m-h}$-distinguished. Hence $\pi$ is $H_{m,m}$-distingusihed by Lemma \ref{lem::preservation-parabolic induction}. If otherwise $\pi_i$ is not one dimensional, we can also move $\pi_i''$ to the beginning of the product \eqref{formula::auxiliary in general case--product} by Lemma \ref{lem::ladder--Commutativity} and our ordering of $\pi_i$'s. By part (2) of Lemma \ref{lem::symmetry},
	\begin{align}\label{formula::auxiliary in general case--representation}
		 \pi_1 \times \cdots \times \pi_{i-1} \times \pi_{i+1}   \times \cdots \times \pi_t  \times (\pi_i'')^{\vee} 
	\end{align}
	is $(H_{m,m-h},\mu_{-h/2})$-distinguished. This is impossible by Proposition \ref{prop::Geometric Lemma--Products Products}, and then we are done. Indeed, firstly, we can check easily that $\pi_i''$, hence its contragradient $(\pi_i'')^{\vee}$, cannot be $(H_{p_1,q_1},\mu_{a_1})$-distinguished for any $(p_1,q_1,a_1)$ by Lemma \ref{lem::distinguished right aligned--shape}. Secondly, note that $\mathbf{e}((\pi_i'')^{\vee}) = \nu^{-h-1/2}  $ is not dual to $\mathbf{b}(\pi_i) = \nu^{h-1/2}$ for all $i$.
\end{proof}

\begin{thm}\label{thm::proof+classification+UnitaryDual}
	Let $\pi $ be an irreducible unitary representation of $G_{2m}$ of Arthur type. Then $\pi$ is $H_{m,m}$-distinguished if and only if $\pi$ is of the form
	\begin{align}\label{formula::standard form-distinguished unitary}
		(\sigma_1 \times \sigma_1^{\vee})  \times \cdots \times (\sigma_r \times \sigma_r^{\vee})  \times \sigma_{r+1} \times \cdots \times \sigma_s.  
	\end{align}
	where each $\sigma_i$ is a Speh representation for $i = 1, \cdots, r$, and each representation $\sigma_j$ is $H_{m_j,m_j}$-distinguished for some positive integer $m_j$, $j =  r+1 , \cdots ,s$.
\end{thm}
\begin{proof}
	By the work of Blanc and Delorme \cite{Blanc-Delorme}, we know that $\sigma_j \times \sigma_j^{\vee}$ is $H_{m_j,m_j}$-distinguished with $m_j$ the degree of $\sigma_j$, $j = 1,\cdots,r$. One direction then follows from Lemma \ref{lem::preservation-parabolic induction}. Write $\pi = \pi_1 \times \cdots \times \pi_t$ to be the Tadi\'{c} decomposition of $\pi$. We prove the other direction by induction on $t$. The case $t = 1$ is obvious. In general, as $\pi$ is a commutative product, we order these $\pi_i$ in the following way: We first group these $\pi_i$ by cuspidal supports. Namely, representations with cuspidal supports contained in the union of one cuspidal line and its contragredient are put in the same group. The ordering of the groups can be arbitrary. For representations within the same group, if their cuspidal supports are contained in one cuspidal line, we arrange the ordering such that when $i < j$, we have either $\mathbf{b}(\pi_i) < \mathbf{b}(\pi_j)$, or $\mathbf{b}(\pi_i) = \mathbf{b}(\pi_j)$ and $\mathbf{ht}(\pi_i) \leqslant \mathbf{ht}(\pi_j)$; if their cuspidal supports are contained in two different cuspidal lines, we arrange the ordering such that when $i < j$, we have $\mathbf{ht}(\pi_i) \leqslant \mathbf{ht}(\pi_j)$.
	
	By our assumption, $\pi$ is $H_{m,m}$-distinguished. We apply Propositon \ref{prop::ladder-products of left aligned} and discuss case by case.
	
	Case A. There exists a division of $\pi_t$, $\pi_t = \pi_t' \sqcup \pi_t''$, where $\pi_t'$ is neither $\pi_t$ nor the trivial representation of $G_0$, such that, among other things, $\pi_t'$ is $(H_{r,s},\mu_{(r-s)/2})$-distinguished for two nonnegative integers $r$ and $s$. We have two subcases.
	\begin{itemize}
		\item[(1)]   The representation $\pi_t'$ is \emph{not one dimensional}. By Proposition \ref{prop::left aligned + distinguish = Speh}, we know $\pi_t'$ is an essentially Speh representation. So, by Corollary \ref{cor::Speh + distinguishe = max Levi G_m,m}, we have $r = s$. That is, $\pi_t'$ is $H_{r,r}$-distinguished. In particular, $\pi_t'$ is self-dual, and hence $\pi_t$ is self-dual. This further shows that $\pi_t'$ is a Speh representation. By Proposition \ref{prop::ladder-products of left aligned}, there exists $i$, $1 \leqslant i \leqslant t-1$, and a division of $\pi_i$, $\pi_i = \pi_i' \sqcup \pi_i''$ such that $(\pi_t'')^{\vee} \cong \pi_i'$ and that
		\begin{align}\label{formula::Classification-Case A-representation}
			\pi_1 \times \cdots \times \pi_{i-1} \times \pi_i'' \times \pi_{i+1} \times \cdots \times \pi_{t-1}
		\end{align}
		is $H_{m',m'}$-distinguished for some positive integer $m'$. Thus we have $\textbf{b}(\pi_t) = \textbf{b}(\pi_i)$. By our assumption on the ordering of representations, we have $\mathbf{ht}(\pi_i) \leqslant \mathbf{ht}(\pi_t)$. As $\pi_t'$ is a self-dual Speh representation that does not equal to $\pi_t$, we have $\mathbf{ht}(\pi_t'')  = \mathbf{ht}(\pi_t)$. As $\mathbf{ht}(\pi_i') \leqslant \mathbf{ht}(\pi_i)$, we have $\mathbf{ht}(\pi_t) = \mathbf{ht}(\pi_i)$ due to  the fact that $(\pi_t'')^{\vee} \cong \pi_i'$. Thus we have $\pi_i \cong \pi_t$ and $\pi_i'' \cong \pi_t'$. Recall that $\pi_t'$ is a $H_{r,r}$-distinguished Speh representation. So, by induction hypothesis, the representation \eqref{formula::Classification-Case A-representation} is of the form \eqref{formula::standard form-distinguished unitary}. After removing $\pi_i''$ in the product, we still get a representation of the form \eqref{formula::standard form-distinguished unitary}. Therefore, by adding $\pi_t \times\pi_i$, we get that $\pi$ is of the form \eqref{formula::standard form-distinguished unitary}.
		\item[(2)]  The representation $\pi_t'$ is \emph{one dimensional}. If $r > 0$ and $s > 0$, then $\pi_t'$ is the trivial representation $\mathbf{1}$ of $G_{2r}$ by Lemma \ref{lem::one dimensional--distinction}. Note that, in this case, $\pi_t$ is not a one dimensional representaion. Then by the same arguments as in Case A (1), we are done in this case. If one of $r$, $s$ is $0$, then $\pi_t'$ is the character $\nu^{h/2}$ of $G_h$, $h = \max\{r,s\}$. Thus we have $\textbf{b}(\pi_t) = \mathbf{b}(\pi_t') = \nu^{h-1/2}$. In particular, $\pi_t$ is self-dual. By Proposition \ref{prop::ladder-products of left aligned}, there exists $i$, $1 \leqslant i \leqslant t-1$, and a division of $\pi_i$, $\pi_i = \pi_i' \sqcup \pi_i''$ such that $(\pi_t'')^{\vee} \cong \pi_i'$ and that
		\begin{align}\label{formula::Main Theorem--Case A (2)}
			\pi_1 \times \cdots \times \pi_{i-1} \times \pi_i'' \times \pi_{i+1} \times \cdots \times \pi_{t-1}
		\end{align}
		is $(H_{m-n_t,m-n_t +h},\mu_{h/2})$-distinguished with $n_t$ the degree of $\pi_t$. Thus we have $\mathbf{b}(\pi_i) = \mathbf{b}(\pi_t) = \nu^{h - 1/2}$, and $\pi_i$ is also self-dual. By our assumption on the ordering of representations, we have $\mathbf{ht}(\pi_i) = \mathbf{ht}(\pi_t)$, and hence $\pi_i \cong \pi_t$. Thus, the representation $\pi_i''$ is the character $\nu^{-h/2}$ of $G_h$. By Lemma \ref{lem::ladder--Commutativity}, the representation \eqref{formula::Main Theorem--Case A (2)} is isomorphic to the representation
		\begin{align*}
			\pi_1 \times \cdots \times \pi_{i-1}  \times \pi_{i+1} \times \cdots \times \pi_{t-1} \times \nu^{-h/2}.
		\end{align*}
		By Lemma \ref{lem::cancellation result}, the representation $\pi_1 \times \cdots \times \pi_{i-1}  \times \pi_{i+1} \times \cdots \times \pi_{t-1}$ is $H_{m-n_t,m-n_t}$-distinguished, and hence is of the form \eqref{formula::standard form-distinguished unitary} by induction hypothesis. Therefore, by adding $\pi_i \times \pi_t$, we get that $\pi$ is of the form \eqref{formula::standard form-distinguished unitary}.
	\end{itemize}

	Case B. In this case the representation $\pi_t$ is $(H_{r,s},\mu_{(r-s)/2})$-distinguished for two nonnegative integers $r$ and $s$, and the representation
	\begin{align*}
		\pi_1 \times \cdots \times \pi_{t-1}
	\end{align*}
	is $(H_{m - r,m-s},\mu_{(r-s)/2})$-distinguished. As $\pi_t$ is a Speh representation, by consideration of its central character, we have $r = s$. Therefore, by induction hypothesis we are done.
	
	Case C. There exists $i$, $1 \leqslant i \leqslant t-1$, and a division of $\pi_i$, $\pi_i = \pi_i' \sqcup \pi_i''$, such that $(\pi_t)^{\vee} \cong \pi_i'$ and that the representation 
	\begin{align*}
		\pi_1 \times \cdots \times \pi_{i-1}  \times  \pi_i'' \times \pi_{i+1} \times \cdots \times \pi_{t-1}
	\end{align*}
	is $H_{m-n_t,m-n_t}$-distinguished. By our assumption on the ordering of representations, we have $\pi_i \cong (\pi_t)^{\vee}$. Thus $\pi_i''$ is the trivial representation of $G_0$. By induction hypothesis, the representation $\pi_1 \times \cdots \times \pi_{i-1}  \times \pi_{i+1} \times \cdots \times \pi_{t-1}$ is of the form \eqref{formula::standard form-distinguished unitary}. Therefore, by adding $\pi_t \times (\pi_t)^{\vee}$, the representation $\pi$ is of the form \eqref{formula::standard form-distinguished unitary}.
\end{proof}

To classify distinguished representations in the entire unitary dual, it remains to consider distinction of complementary series representations. Recall that a complementary series representation is an irreducible unitary representation of the form $\nu^{\alpha}\speh(\delta,k) \times \nu^{-\alpha}\speh(\delta,k)$ with $ 0 < \alpha < 1/2$, and is denoted by $\speh(\delta,k)[\alpha,-\alpha]$. By the work of Blanc and Delorme \cite{Blanc-Delorme}, one sees that $\speh(\delta,k)[\alpha,-\alpha]$ is $H_{m,m}$-distinguished if and only if it is self-dual, where $m$ is the degree of $\speh(\delta,k)$. To apply the geometric lemma, we first note the following lemma.

\begin{lem}\label{lem::classification--entire unitary dual}
	Let $\rho$ be a unitary supercuspidal representation of $G_d$ and $c$ a fixed integer. Let $\pi$ be a ladder representation of $G_n$ with cuspidal supports contained in the cuspidal line $\BZ \nu^{\alpha+c/2}\rho$ or $\BZ \nu^{-\alpha+c/2}\rho$ with $0 < \alpha < 1/2$, then $\pi$ cannot be self-dual. If, moreover, $\pi$ is left aligned, then $\pi$ cannot be $(H_{p,q},\mu_{(p-q)/2})$-distinguished for certain nonnegative integers $p$, $q$ with $p+q=n$.
\end{lem}
\begin{proof}
	As $0 < \alpha < 1/2$, the cuspidal line  $\BZ \nu^{\alpha+c/2}\rho$ (or $\BZ \nu^{-\alpha+c/2}\rho$) is not self-dual. Thus $\pi$ cannot be self-dual by Lemma \ref{lem::DualOfLadder}. For the second statement, if $\pi$ is one dimensional, then by Lemma \ref{lem::one dimensional--special condition}, the cuspidal supports of $\pi$ is contained in $\BZ \nu^0$ or $\BZ \nu^{-1/2}$. This contradicts with our assumption; if $\pi$ is not one dimensional, then by Proposition \ref{prop::left aligned + distinguish = Speh} and Corollary \ref{cor::Speh + distinguishe = max Levi G_m,m}, one sees $\pi$ is self-dual. This is absurd as shown by the first statement.
\end{proof}

\begin{thm}
	An irreducible unitary representation $\pi$ of $G_{2n}$ is $H_{n,n}$-distinguished if and only if it is self-dual and its Arthur part $\pi_{\textup{Ar}}$ is of the form \eqref{formula::standard form-distinguished unitary}.
\end{thm}
\begin{proof}
	To simplify notation, we will say a representaion $H$-distinguished for $H_{m,m}$-distinguished when there is no need to address $m$. Write $\pi = \pi_{\textup{Ar}} \times \pi_c$. If $\pi$ is self-dual, by uniqueness of Tadi\'{c} decomposition, we have $\pi_c$ is also self-dual. As $\pi_c$ is a commutative product of complementary series representations, we have $\pi_c$ is $H$-distinguished. The `if' part then follows from Lemma \ref{lem::preservation-parabolic induction}. For the `only if' part, write $\pi$ as a product of essentially Speh representations
	\begin{align}\label{formula::classification--entireUnitaryDual+products}
		\pi_1 \times \cdots \times \pi_t \times \nu^{\alpha_1}\speh(\delta_1,k_1) \times \nu^{-\alpha_1}\speh(\delta_1,k_1) \times \cdots \times \nu^{\alpha_r}\speh(\delta_r,k_r) \times \nu^{-\alpha_r}\speh(\delta_r,k_r)
	\end{align}
	such that $k_1 \leqslant k_2 \leqslant \cdots \leqslant k_r$, and that $\pi_i$ is a Speh representation for $i = 1, \cdots,t$. Now we appeal to Proposition \ref{prop::ladder-products of left aligned}. By Lemma \ref{lem::classification--entire unitary dual}, only Case C can happen. Note that we have  $k_1 \leqslant \cdots \leqslant k_r$ and $0 < \alpha_i < 1/2$, $i = 1 ,\cdots,r$. By simple arguments we can show that each time after applying Proposition \ref{prop::ladder-products of left aligned}, we can delete two non-unitary essentially Speh representations in the product \eqref{formula::classification--entireUnitaryDual+products}, and the new representation is $H$-distinguished. Thus by a repeated use of Proposition \ref{prop::ladder-products of left aligned}, we get $\pi_{\textup{Ar}} = \pi_1 \times \cdots \times \pi_t$ is $H$-distinguished. The `only if' part then follows from Theorem \ref{thm::proof+classification+UnitaryDual}. 
\end{proof}

\bibliographystyle{spbasic}      % basic style, author-year citations
%\bibliography{references}  % name your BibTeX data base

\end{document}